\documentclass[preprint]{elsarticle}

\usepackage{amsmath}
\usepackage{latexsym,amsmath}
\usepackage{amsthm}
\usepackage{amssymb}
\usepackage{enumerate}

\usepackage{epsfig, verbatim}
\usepackage{fullpage}
\usepackage{graphicx}
\usepackage{color}

\newtheorem{theorem}{Theorem}
\newtheorem{lemma}[theorem]{Lemma}

\newtheorem{corollary}[theorem]{Corollary}
\newtheorem{conjecture}[theorem]{Conjecture}

\newtheorem{claim}{Claim}

\newtheorem{obs}[theorem]{Observation}

\def\soft#1{\leavevmode\setbox0=\hbox{h}\dimen7=\ht0\advance
    \dimen7 by-1ex\relax\if t#1\relax\rlap{\raise.6\dimen7
    \hbox{\kern.3ex\char'47}}#1\relax\else\if T#1\relax
    \rlap{\raise.5\dimen7\hbox{\kern1.3ex\char'47}}#1\relax
    \else\if d#1\relax\rlap{\raise.5\dimen7\hbox{\kern.9ex
    \char'47}}#1\relax\else\if D#1\relax\rlap{\raise.5\dimen7
    \hbox{\kern1.4ex\char'47}}#1\relax\else\if l#1\relax
    \rlap{\raise.5\dimen7\hbox{\kern.4ex\char'47}}#1\relax
    \else\if L#1\relax\rlap{\raise.5\dimen7\hbox{\kern.7ex
    \char'47}}#1\relax\else\message{accent \string\soft
    \space #1 not defined!}#1\relax\fi\fi\fi\fi\fi\fi}

\title{VC-dimension and Erd\H{o}s-P\'osa property}

\author[mtp]{Nicolas Bousquet\footnote{lastname@lirmm.fr}}

\author[lyon]{St\'ephan Thomass\'e\footnote{firstname.lastname@ens-lyon.fr. Partially supported by ANR Project STINT under Contract ANR-13-BS02-0007.}}

\address[mtp]{Universit\'e Montpellier 2, 161 rue Ada, 34392 Montpellier Cedex 5 - France.}
\address[lyon]{LIP, ENS Lyon, 46 all\'ee d'Italie, 69364 Lyon Cedex 07 - France.}

\begin{document}

\begin{abstract}
Let $G=(V,E)$ be a graph. A {\it $k$-neighborhood} in $G$ is a set
of vertices consisting of all the vertices at distance at most $k$ from 
some vertex of $G$. The hypergraph on vertex set $V$ which edge
set consists of all the $k$-neighborhoods of $G$ for all $k$ is the {\it neighborhood 
hypergraph} of $G$. Our goal in this paper is to investigate 
the complexity of a graph in terms of its neighborhoods.
Precisely, we define the {\it distance VC-dimension} of a graph $G$ 
as the maximum taken over all induced subgraphs $G'$ of $G$ of the VC-dimension 
of the neighborhood hypergraph of $G'$. 
For a class of graphs, having bounded distance VC-dimension both generalizes 
minor closed classes 
and graphs with bounded clique-width. 

Our motivation is a result of Chepoi, Estellon and Vax\`es~\cite{ChepoiEV07} asserting
that every planar graph of diameter $2\ell$ can be covered by a bounded 
number of balls of radius $\ell$. In fact, they obtained the existence of a function $f$ such that 
every set $\cal F$ of balls of radius $\ell$ in a planar graph
admits a hitting set of size $f(\nu)$ where $\nu$ is the maximum number 
of pairwise disjoint elements of $\cal F$. 


Our goal is to generalize the proof of~\cite{ChepoiEV07} with the unique 
assumption of bounded distance VC-dimension of neighborhoods. In other words, the set of balls 
of fixed radius in a graph with bounded distance VC-dimension has the Erd\H{o}s-P\'osa property.
\end{abstract}


\maketitle


\section{Introduction}

\paragraph{$B$-hypergraph and dominating sets}
Let $G=(V,E)$ be a graph. A \emph{dominating set of $G$} is a set $X$ of vertices such that for every vertex $v$, there exists a vertex $x \in X$ satisfies either $x=v$ or $v$ is a neighbor of $x$. 
In other words, all the vertices of $V$ are at distance at most one from a vertex of $X$. In this paper we focus on a generalization of dominating sets called dominating sets at distance $\ell$. A set $X$ is a \emph{dominating set at distance $\ell$} if every vertex of the graph is at distance at most $\ell$ from a vertex of $X$.

A hypergraph is a pair $(V,F)$ where $V$ is a set of vertices and $F$ is a set of subsets of $V$ called hyperedges. 
For the study of dominating sets, a natural hypergraph arises: the $B_1$-hypergraph. The \emph{$B_1$-hypergraph} of $G$ has vertex set $V$ and hyperedges are the closed neighborhoods of the vertices of the graph. 
Since we consider neighborhoods at distance $\ell$ in this paper, we naturally generalize the $B_1$-hypergraph into the $B_\ell$-hypergraph by replacing closed neighborhoods by balls of radius $\ell$ centered in every vertex of the graph. The \emph{$B$-hypergraph} is the edge-union of the $B_\ell$-hypergraphs for all $\ell$.

A \emph{hitting set} of a hypergraph $H=(V,F)$ is a subset of vertices intersecting every hyperedge. In other words, it is a subset $X$ of vertices such that for every $e \in F$, $e \cap X \neq \emptyset$. One can note that a hitting set of the $B_\ell$-hypergraph of a graph $G$ is a dominating set at distance $\ell$ of the graph $G$ (and the converse also holds). Indeed, let $X$ be a hitting set of the $B_\ell$-hypergraph $H$ of $G$. For every vertex $v \in V$, there exists $x \in X$ such that $x$ is at distance at most $\ell$ from $v$. So the whole set of vertices is at distance at most $\ell$ from a vertex of $X$, \emph{i.e.} $X$ is a dominating set at distance $\ell$ of $G$. In the following we focus on hitting sets of the $B_\ell$-hypergraphs. The minimum size of a hitting set, denoted by $\tau$, is called the \emph{transversality}. The \emph{packing number}, denoted by $\nu$, is the maximum number of pairwise disjoint hyperedges.

\paragraph{Complexity of graphs and VC-dimension} 
\begin{figure}
\centering
\includegraphics[scale=0.4]{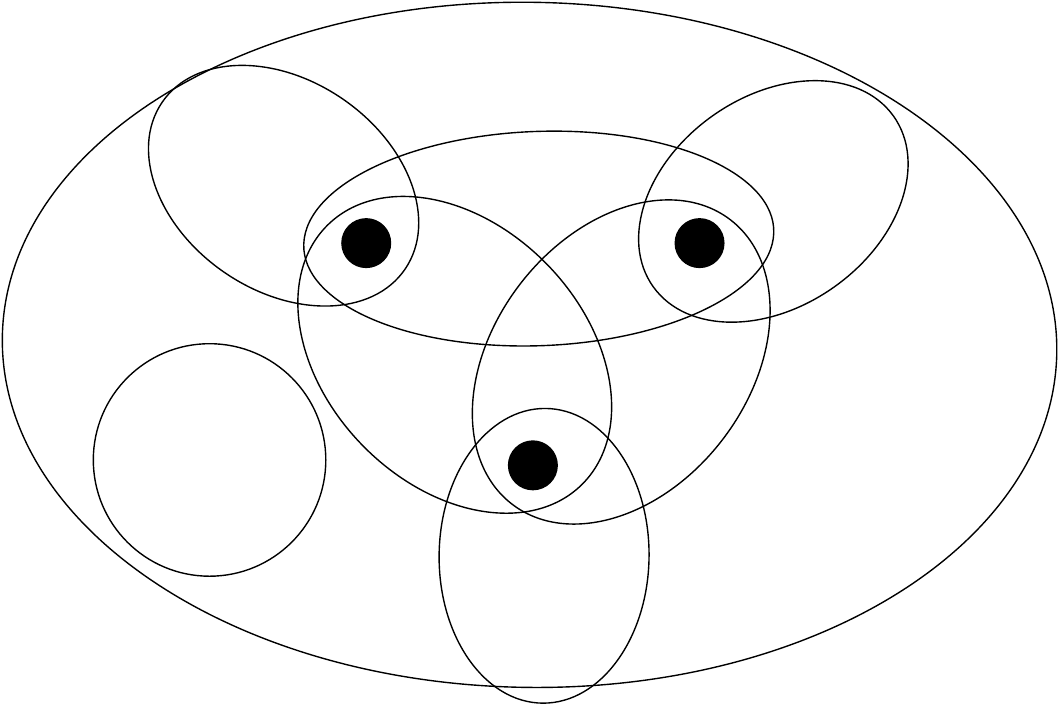}
\caption{A shattered set of size $3$.}
\label{taille3shattered}
\end{figure}

A set $X$ of vertices is \emph{shattered} (resp. \emph{$2$-shattered}) if for every subset $X'$ of $X$ (resp. every subset of $X$ of size $2$) there exists a hyperedge $e$ such that $e \cap X = X'$ (see Figure~\ref{taille3shattered}). Introduced in~\cite{Sauer72,VapnikC71}, the \emph{Vapnik-Chervonenkis dimension} (or \emph{VC-dimension} for short) (resp. \emph{$2$VC-dimension}) of a hypergraph $H$ is the maximum size of a shattered set (resp. $2$-shattered set).
It is a good complexity measure of a hypergraph, for instance in the learnability sense. A bounded VC-dimension provides upper bounds on  the number of hyperedges~\cite{Sauer72} but also on the transversality~\cite{DingSW94,HausslerW86,Matousek04}. The VC-dimension has many applications, in learnability theory \cite{HausslerW86} and in computational geometry \cite{ChazelleW89}. More recently, several applications were developed in graph theory, see~\cite{AlonB06,ChepoiEV07,LuczakT10} for instance.

One of our goals was to extend this notion on graphs to catch the complexity of a graph at large distance. 
The distance VC-dimension of a graph $G$ could be defined as the VC-dimension of the $B$-hypergraph of the graph $G$.
Since throughout this paper we only consider graphs closed under induced subgraphs, we define the \emph{distance VC-dimension of a graph} $G$ (resp. \emph{distance $2$VC-dimension} of the graph $G$) is the maximum over induced subgraphs of the distance VC-dimension (resp. $2$VC-dimension) of the $B$-hypergraph. 
Since the VC-dimension ``measures'' the local randomness of hypergraphs, it is natural to think that classes with a lot of structure might have a bounded VC-dimension. 
In Section~\ref{sec:classesboundedvc}, we prove that two famous graph classes have bounded distance VC-dimension. First we show that the class of $K_n$-minor free graphs has distance VC-dimension at most $n-1$. The proof is almost the proof of Chepoi, Estellon and Vax\`es that the $B_\ell$-hypergraph of planar graphs has distance VC-dimension at most $4$\footnote{In their paper, Chepoi, Estellon and Vax\`es noted that their proof for planar graphs can be extended to $K_n$-minor free graphs.}. Then we show that the class  of bounded rankwidth graphs have bounded distance VC-dimension. Actually, we prove a slightly stronger statement for these two classes: their distance $2$VC-dimension is bounded.  We finally provide some graphs of bounded distance VC-dimension with an arbitrarily large distance $2$VC-dimension.

\paragraph{Erd\H{o}s-P\'osa property}  
Chepoi, Estellon and Vax\`es \cite{ChepoiEV07} proved that every planar graph of diameter $2\ell$ can be covered by $c$ balls of radius $\ell$ (where $c$ does not depend on $\ell$). It answered a conjecture of Gavoille, Peleg, Raspaud and Sopena~\cite{GavoillePRS01}. Their proof uses the concept of VC-dimension but also planarity of the graph. One of our aims was to determine if the planarity arguments are necessary or if a purely combinatorial proof of this result exists. 

Let $G$ be a graph. We denote by respectively $\nu_\ell$ and $\tau_\ell$ the packing number and the transversality of the $B_\ell$-hypergraph of $G$.
Note that the $B_\ell$-hypergraph of a planar graph of diameter $2 \ell$ satisfies $\nu_\ell=1$. Indeed for every $u,v \in V$, since the diameter of the graph is at most $2\ell$, there exists a vertex $x$ at distance at most $\ell$ from both $u$ and $v$, so the hyperedges centered in $u$ and in $v$ intersect. Since $\tau_\ell$ equals the minimum size of a dominating set at distance $\ell$, we have $\tau_\ell \geq \nu_\ell$.
A class of hypergraphs such that the transversality of every hypergraph is bounded by a function of its packing number is said to satisfy the \emph{Erd\H{o}s-P\'osa property} (and the function is called the \emph{gap function}). In their seminal paper~\cite{ErdosP65}, Erd\H{o}s and P\'osa proved that the minimum size of a feedback vertex set can be bounded by a function of the the maximum number of vertex disjoint cycles: differently the cycle hypergraph of $G$ has the erd\H{o}s-P\'osa property (the vertices of the hypergraph are the vertices of the graph and the hyperedges are the cycles of the graph). 

In Section~\ref{VC}, we first simplify and generalize the proof of Chepoi, Estellon and Vax\`es. More precisely we prove that the $B_\ell$-hypergraph of any graph $G$ has a dominating set at distance $\ell$ of size $\mathcal{O}(\nu_\ell^{2d+1})$ where $d$ denotes the distance $2$VC-dimension of $G$. Note that the function depends on $\nu_\ell$ but not directly on $\ell$. Since planar graphs have distance $2$VC-dimension at most $4$, it ensures that the $B_\ell$ hypergraph of any planar graph of diameter $2\ell$ satisfies $\tau_\ell \leq 35200$. There is no doubt that this upper bound is still far away from the optimal one.
For small diameters, better bounds exist. For instance every planar graph of radius $2$ has a dominating set of size at most $3$~\cite{GoddardH02,MacGillivrayS96}. 

Since some graphs of bounded distance VC-dimension have an arbitrarily large distance $2$VC-dimension, it raises a natural question: is it possible to extend this result on graphs of bounded distance $2$VC-dimension to graphs of bounded distance VC-dimension. Section~\ref{VC} consists in proving that the answer to this question is positive.
More formally, we prove that there exists a function $f$  such that the $B_\ell$-hypergraph of a graph of distance VC-dimension $d$ has a hitting set of size at most $f(\nu_{\ell},d)$. The original proof of Chepoi, Estellon and Vax\`es for planar graphs is based on the same method but they conclude using topological properties of planar graphs. Since we only deal with combinatorial structures, our proof is more technically involved. Note nevertheless that the function $f$ is exponential in the distance VC-dimension while the one provided by the distance $2$VC-dimension is polynomial.

We will finally close this paper by some concluding remarks and open problems on distance VC-dimension and Erd\H{o}s-P\'osa property.

\section{Preliminaries}
It is sometimes convenient to see a hypergraph as its \emph{incidence bipartite graph} $B_H$ with vertex set $V \cup E$ in which there is an edge 
between $x \in V$ and $e \in E$ iff $x \in e$. Note that the pair $(V,E)$ is oriented, and the hypergraph associated to 
the pair $(E,V)$ is called the \emph{dual hypergraph}. The vertices of the dual hypergraph are the hyperedges of the original 
one, and the hyperedges of the dual hypergraph are the subsets of $E$ containing the vertex $v$, for every $v$.
%
The \emph{dual VC-dimension of $H$} is the VC-dimension of the dual hypergraph of $H$. The VC-dimension of $H$ and the dual VC-dimension of $H$ are equivalent 
up to an exponential function~\cite{Assouad83}.
Similarly, the \emph{dual $2$VC-dimension of $H$} is the $2$VC-dimension of the dual hypergraph of $H$. The 
$2$VC-dimension is larger than or equal to the VC-dimension and the gap can be arbitrarily large. Indeed, 
consider the clique $K_n$. Its $2$VC-dimension is equal to $n$ whereas its VC-dimension is at most $2$ (since 
no hyperedge contains $3$ vertices). The same example ensures that no function links $2$VC-dimension and dual $2$VC-dimension.

A \emph{transversal set} (or \emph{hitting set}) of a hypergraph $H$ is a set of vertices intersecting each hyperedge. The \emph{transversality} $\tau$ of a hypergraph is the minimum size of a 
transversal set. The \emph{packing number} $\nu$ of a hypergraph is the maximum number of vertex disjoint hyperedges. 
A class of hypergraphs $\mathcal{H}$ has the \emph{Erd\H{o}s-P\'osa property} if there exists a function $f$ such that for 
all $H \in \mathcal{H}$, $\tau \leq f(\nu)$.
We denote by $\nu_{\ell}$ and $\tau_{\ell}$ respectively the packing number and the transversality of the $B_\ell$-hypergraph of $G$.
Note that the $B_\ell$-hypergraph of a graph $G$ and its dual are the same since for every pair of vertices $x,y$,
$x \in B(y,\ell)$ if and only if $y \in B(x,\ell)$. So:

\begin{obs}\label{obs:Blisomdual}
 The $B_\ell$-hypergraph is isomorphic to the dual of the $B_\ell$-hypergraph.
\end{obs}

Let $G=(V,E)$ be a graph. Let $X \subseteq V$. The \emph{graph induced by $X$} is the graph on vertex set $X$ whose edges are edges of $G$ with both endpoints in $X$. A \emph{walk} of length $k$ from $x \in V$ to $y \in V$ is a sequence of vertices 
$x=x_0,x_1,\ldots,x_{k-1},x_k=y$ where $x_ix_{i+1} \in E$ for each $0\leq i \leq k-1$. A \emph{path} is a walk 
with pairwise distinct vertices. The vertices $x$ and $y$ are the \emph{endpoints} of the walk. 
The \emph{$x_ix_j$-subpath} is the path $x_i,x_{i+1},\ldots,x_j$. The \emph{neighbors} of the vertex $x_i$ 
on the path are the vertices $x_{i-1}$ and $x_{i+1}$ whenever they exist. A 
\emph{minimum path} from $x$ to $y$, also called \emph{minimum $xy$-path}, is a path of minimum length 
from $x$ to $y$. The \emph{distance between $x$ and $y$}, denoted by $d(x,y)$ is the length of a minimum $xy$-path
when such a path exists and $+ \infty$ otherwise. The \emph{distance} between a set $X$ and a set $Y$
is the minimum for all $x,y \in X \times Y$ of the distance between $x$ and $y$.
The \emph{ball of center $x$ and radius $k$}, denoted by $B(x,k)$, is the set of vertices at distance 
at most $k$ from $x$. The \emph{neighbors} of $x$, denoted by $N(x)$ are the vertices of $B(x,1)$ 
distinct from $x$.



Let us conclude this section by an observation which ensures that we can restrict our study to connected subgraphs:
\begin{obs}\label{obs:connectedgraphs}
 The distance VC-dimension of a non connected graph is the distance VC-dimension of the maximum of its connected components.
\end{obs}

\section{Graphs of bounded distance VC-dimension}\label{sec:classesboundedvc}

In this section we prove that $K_n$ minor-free graphs and bounded rank-width graphs have bounded distance $2$VC-dimension.
In addition we provide a class of graphs with arbitrarily large distance $2$VC-dimension and distance VC-dimension at most $18$.

\subsection{$K_d$-minor-free graphs have bounded distance VC-dimension}\label{knmin}
A graph $H$ is a \emph{minor} of $G$ if $H$ can be obtained from $G$ by contracting edges, 
deleting edges, and deleting vertices. Theorem~\ref{kdminor} 
is roughly Proposition 1 of \cite{ChepoiEV07}. Since our definitions and statements are slightly different, 
we prove it for the sake of completeness. We first prove an easy lemma before stating the main theorem of
this section.

\begin{lemma}\label{ballcontain}
If $z$ is on a minimum $xy$-path, the ball $B(z,d(x,z))$ is included in $B(y,d(x,y))$.
\end{lemma}
\begin{proof}
Since $z$ is on a minimum $xy$-path, $d(x,y)=d(x,z)+d(z,y)$. Hence $B(y,d(y,z))$ contains $z$ and 
then $B(y,d(y,z)+d(z,x))$ contains $B(z,d(x,z))$.
\end{proof}

\begin{theorem}\label{kdminor}
A $K_d$-minor-free graph has distance $2$VC-dimension at most $d-1$.
\end{theorem}
\begin{proof}
Let $G$ be a graph with distance $2$VC-dimension $d$. Let $X= \{x_1,x_2,\ldots,x_d \}$ be a set of vertices of $G$ which is 
$2$-shattered by the hyperedges of the $B$-hypergraph of $G$. Hence, for every pair $(i,j)$, there exists
a vertex $c_{i,j}$ and an integer $r_{i,j}$ such that $B(c_{i,j},r_{i,j}) \cap X =\{ x_i$, $x_j \}$.
We assume moreover that $r_{i,j}$ is minimum for all choices of $(c_{i,j},r_{i,j})$. A {\it central path}
$P_{i,j}$ is the concatenation of a minimum path from $x_i$ to $c_{i,j}$ and a minimum path from 
$c_{i,j}$ to $x_j$.

\begin{claim}\label{midvertex}
A central path is indeed a path. 
\end{claim}
\begin{proof}
Assume by contradiction that $x$ appears more than once in a central path $P_{i,j}$. Since $P_{i,j}$ is a concatenation 
of a shortest $x_ic_{i,j}$-path and a shortest $c_{i,j}x_j$-path, $x$ appears once between $x_i$ and $c_{i,j}$ and once 
between $c_{i,j}$ and $x_j$. Let us call $Q_1$ the subpath of $P_{i,j}$ from $x$ to $c_{i,j}$ and 
$Q_2$ the subpath of $P_{i,j}$ from $c_{i,j}$ to $x$. Note that $Q_1$ and $Q_2$ are both shortest
paths connecting $c_{i,j}$ and $x$, hence replacing $Q_2$ by the mirror of $Q_1$ gives another 
central path $P'_{i,j}$. The two neighbors of $c_{i,j}$ in $P'_{i,j}$ are 
the same vertex $v$, contradicting the minimality of $r_{i,j}$ 
since $B(v,r_{i,j}-1) \cap X =\{ x_i$, $x_j \}$.
\end{proof}

\begin{claim}\label{kdminlemma}
If $x$ belongs to two distinct central paths, then these paths are $P_{i,j}$ and $P_{i,l}$, and we both have
$d(x,x_i) < d(x,x_j)$ and $d(x,x_i) < d(x,x_l)$.
\end{claim}
\begin{proof}
Assume that $x$ appears in $P_{i,j}$ and $P_{k,l}$, where $d(x,x_i) \leq  d(x,x_j)$ and $d(x,x_k) \leq  d(x,x_l)$.
Free to exchange the roles of $P_{i,j}$ and $P_{k,l}$, we can also assume that $d(x,x_k) \leq d(x,x_i)$. By 
Lemma~\ref{ballcontain}, $x_k \in B(c_{i,j},r_{i,j})$, hence we have $x_k=x_i$ or $x_k=x_j$. Since 
$d(x,x_k) \leq d(x,x_i)\leq  d(x,x_j)$ and $x_k$ is either $x_i$ or $x_j$, we have $d(x,x_k) = d(x,x_i)$.
Hence $d(x,x_i) \leq d(x,x_k)$, and by the same argument, we have $x_i=x_k$ or $x_i=x_l$. Since the central paths
are distinct, we necessarily have $x_i=x_k$. Observe that $d(x,x_i)=d(x,x_j)$, hence $d(x,x_j) \leq   d(x,x_i)$, would give
by the same argument  $x_j=x_k$, hence a contradiction since we would have $x_i=x_j$. 
Therefore $d(x,x_i) < d(x,x_j)$, and for the same reason $d(x,x_i) < d(x,x_l)$.
\end{proof}

Let us now construct some connected subsets $X_i$ for all $1 \leq i \leq d$. For every 
path $P_{i,j}$, the vertices of $P_{i,j}$ closer to $x_i$ than to $x_j$ are added to $X_i$,
the vertices of $P_{i,j}$ closer to $x_j$ than to $x_i$ are added to $X_j$, and the midvertex
(if any) is arbitrarily added to $X_i$ or to $X_j$. 

The crucial fact is that the sets $X_i$ are pairwise disjoint. Indeed, by Claim~\ref{midvertex} and Claim \ref{kdminlemma},
if a vertex $x$ appears in two distinct central paths, these are $P_{i,j}$ and $P_{i,l}$, where 
$d(x,x_i) < d(x,x_j)$ and $d(x,x_i) < d(x,x_l)$. In particular $x$ belongs in both cases to $X_i$.

By construction, the sets $X_i$ are connected and there is always an edge between
$X_i$ and $X_j$ since their union contains $P_{i,j}$. Therefore if the distance $2$VC-dimension is 
at least $d$, the graph contains $K_d$ as a minor.
\end{proof}

\subsection{Bounded rankwidth graphs have bounded distance VC-dimension}\label{secrank}
Let us first recall the definition of rankwidth, introduced by Oum and Seymour in \cite{OumSeymour06}. Let 
$G=(V,E)$ be a graph and  $(V_1,V_2)$ be a partition of $V$. Let $M_{V_1,V_2}$ be the matrix of size 
$|V_1| \times |V_2|$ such that the entry $(x_1,x_2) \in V_1 \times V_2$ equals $1$ if $x_1x_2 \in E$ 
and $0$ otherwise. The \emph{cutrank} $cr(V_1,V_2)$ of $(V_1,V_2)$ is the rank of the matrix $M_{V_1,V_2}$ 
over the field $\mathbb{F}_2$. A \emph{ternary tree} is a tree with nodes of degree $3$ or $1$. The 
nodes of degree $3$ are the \emph{internal nodes}, the other nodes being the \emph{leaves}. 
A \emph{tree-representation of $G$} is a pair $(T,f)$ where $T$ is a ternary tree with $|V|$ leaves 
and $f$ is a bijection from $V$ to the set of leaves. Every edge $e$ of $T$ defines a partition 
of the leaves of $T$. Therefore it defines a partition of the 
vertex set $V$ into $(V_1^e,V_2^e)$. The \emph{rankwidth} $rw$ of a graph $G$ is defined by:
$$rw(G)=\min_{(T,f)} \max_{e \in E(T)} cr(V_1^e,V_2^e)$$

Before stating the main result, let us first state two lemmas concerning rankwidth and ternary trees.

\begin{lemma}\label{patates}
Let $G=(V,E)$ be a graph of rankwidth $k$ and $X,Y$ be the partition of $V$ induced by an edge of a 
tree-representation of $G$ of cutrank $k$. There exist partitions of $X$ and $Y$ into at most $2^k$ sets 
$X_1,\ldots,X_{2^k}$ and $Y_1,\ldots,Y_{2^k}$ such that for all $i,j$, $(X_i \times Y_j)\cap E= \varnothing$ 
or $(X_i \times Y_j)\cap E= X_i \times Y_j$.
\end{lemma}
\begin{proof}
Let $T$ be a tree representation of $G$ of cutrank at most $k$. Let $e$ be an edge of the tree representation of $G$ and $(X,Y)$ be the partition of $V$ induced by $e$. Since 
the cutrank is at most $k$, the matrix $M_{X,Y}$ has rank at most $k$. Hence there exists $j \leq k$ rows $R_1,\ldots, R_j$ which form a base of the rows of the matrix $M_{X,Y}$. By definition, every row corresponds to the neighborhood of a vertex of $X$ into $Y$. Let us denote by $x_i$ the vertex corresponding to $R_i$. We denote by $\mathcal{B}$ the set $\{x_1,\ldots, x_j\}$. \\
For every $\mathcal{B}' \subseteq \mathcal{B}$, $X(\mathcal{B}')$ denotes the subset of $X$ which contains $x$ if $N(x)\cap Y =_{\mathbb{F}_2} \sum_{x_i \in \mathcal{B}'} N(x_i)$. It induces a partition of $X$ since $N(x_1), \ldots, N(x_j)$ is a base of the neighborhoods of $X$ in $Y$. Note that by definition all the vertices of $X(\mathcal{B}')$ have the same neighborhood in $Y$. Observe that a vertex $x \in X(\mathcal{B}')$ is connected to a vertex $y$ iff an odd number of vertices of $\mathcal{B}'$ are connected to $y$. \\ 
For every $\mathcal{B}' \subseteq \mathcal{B}$, $Y(\mathcal{B}')$ is the subset of $Y$ containing $y$ if $N(y) \cap \mathcal{B} = \mathcal{B}'$. It induces a partition of $Y$ into at most $2^j$ sets with the same neighborhood in $\mathcal{B}$.

Let us finally prove that the partitions of $X(\mathcal{B}')_{\mathcal{B}' \subseteq \mathcal{B}}$ and $Y(\mathcal{B}')_{\mathcal{B}' \subseteq \mathcal{B}}$ satisfy the required properties. Let $x,y$ be in $X(\mathcal{B}')\times Y(\mathcal{B}'')$ such that $xy$ is an edge. Since $xy$ is an edge, an odd number of vertices of $\mathcal{B}'$ are connected to $y$. Since all the vertices of $Y(\mathcal{B}'')$ have the same neighborhood in $\mathcal{B}$, all the vertices of  $Y(\mathcal{B}'')$ have an odd number of neighbors on $\mathcal{B}'$. Thus $x$ is connected to all the vertices of $Y(\mathcal{B}'')$. Since all the vertices of $X(\mathcal{B}')$ have the same neighborhood in $Y$, $(X(\mathcal{B}'),Y(\mathcal{B}''))$ forms a complete bipartite graph.
\end{proof}

\begin{lemma}\label{separationtree}
Every ternary tree $T$ with $\alpha>2$ labeled leaves has an edge $e$ such that the partition induced by 
$e$ has at least $\alpha/3$ labeled leaves in both of its two connected components.
\end{lemma}
\begin{proof}
Orient every edge of $T$ from the component 
with less labeled leaves to the other one (when equality holds, orient arbitrarily). Observe 
that leaves are sources of this oriented tree. Let $v$ be an internal node of $T$
which is a sink. Consider a component $C$ of $T\setminus v$ with at least $\alpha/3$
labeled leaves. Call $e=vw$ the edge of $T$ inducing the partition $(T\setminus C,C)$. 
Since $e$ is oriented from $w$ to $v$, the component $T\setminus C$ has at least $\alpha/2$
labeled leaves, thus $e$ is the edge we are looking for.
\end{proof}

\begin{theorem}\label{rankwidthborne}
The distance $2$VC-dimension of a graph with rankwidth $k$ is at most $3\cdot 2^{k+1}+2$.
\end{theorem}
\begin{proof} 
Assume by contradiction that the $B$-hypergraph of a graph $G$ of rankwidth $k$ admits a $2$-shattered set 
$S$ of size $3(2^{k+1}+1)$. Let $(T,f)$ be a tree decomposition of $G$ achieving rankwidth $k$. By 
Lemma~\ref{separationtree}, there is an edge $e$ of $T$ such that the partition induced 
by $e$ has at least $2^{k+1}+1$ vertices of $S$ in both connected components. Let $V_1,V_2$ 
(resp. $X,Y$) be the partition of $V$ (resp. $S$) induced by $e$. Let $x_1,\ldots,x_{2^{k+1}+1}$ 
and $y_1,\ldots,y_{2^{k+1}+1}$ be distinct vertices of $X$ and $Y$ respectively.

Since $S$ is $2$-shattered, for each $(x_i,y_j)  \in X \times Y$, there is a ball \emph{$B_{i,j}$} 
such that $B_{i,j} \cap S =\{x_i,y_j\}$ where $B_{i,j}$ is chosen with minimum radius.

\begin{claim}\label{rankchoix}
One of the following holds:
\begin{itemize}
\item There is an $i$ such that at least $2^k+1$ balls $B_{i,j}$ have their centers in $V_1$.
\item There is a $j$ such that at least $2^k+1$ balls $B_{i,j}$ have their centers in $V_2$.
\end{itemize}
\end{claim}
\begin{proof}
Orient the edges of the complete bipartite graph with vertex set $X \cup Y$ such that $x_i\rightarrow y_j$ if $B_{i,j}$ 
has its center in $V_1$ and $x_i\leftarrow y_j$ otherwise. The average out-degree of the vertices of $X \cup Y$ is $2^k+\frac 12$.
So a vertex has out-degree at least $2^k+1$. \\
Assume that the vertex $x_i \in X$ has out-degree at least $2^k+1$. There exist $2^k+1$ vertices of $Y$, w.l.o.g. $y_1,\ldots,y_{2^k+1}$, such that
$x_iy_1,\ldots,x_iy_{2^k+1}$ are arcs. So the balls $B_{i,j}$ have their centers in $V_1$
for all $j \in \{ 1,\ldots,2^k+1\}$, and then the first point holds.
If a vertex of $Y$ has out-degree at least $2^k+1$, a symmetric argument ensures that the second point holds, which achieves the proof.
\end{proof}

By Claim \ref{rankchoix}, we can assume without loss of generality that 
$B(1,1),B(1,2), \ldots, B(1,2^k+1)$ have their centers in $V_1$. We denote 
by $c_i$ and $r_i$ respectively the center and the radius of $B(1,i)$ and by $P_i$ a minimum $c_iy_i$-path. By 
the pigeonhole principle, two $P_i$'s leave $V_1$ 
by the same set of vertices given by the partition of Lemma~\ref{patates}. 
Without loss of generality, we assume that these paths are $P_1$ and $P_2$ and we denote by $z_1$ and $z_2$ respectively their last vertices in $V_1$. 
We finally assume that $d(z_1,y_1) \leq d(z_2,y_2)$. 
By Lemma~\ref{ballcontain}, the ball $B(z_2,d(z_2,y_2))$ is included in $B(c_2,r_2)$ since $z_2$ 
is on a minimum path from $c_2$ to $y_2$. Let $z_1z'_1$ be the first edge of 
$P_1$ between $z_1$ and $y_1$ (hence $z'_1$ belongs to $Y$). By 
Lemma~\ref{patates}, $z'_1$ is also a neighbor of $z_2$ since $z_1$ and 
$z_2$ have the same neighborhood in $Y$. Thus $y_1 \in B(z_2,d(z_2,y_2))$. Thus 
$y_1 \in B(z_2,d(z_2,y_2))$ which contradicts the hypothesis.
\end{proof}

Since the rankwidth is equivalent, up to an exponential function, to the cliquewidth~\cite{OumSeymour06}, Theorem~\ref{rankwidthborne} implies that every class of graphs with bounded clique-width has bounded distance $2$VC-dimension.

\subsection{Unbounded distance $2$VC-dimension with bounded distance VC-dimension}
\begin{theorem}\label{unbounded2VC}
Let $n,\ell$ be two integers. There exists a graph $G_{n,\ell}$ of distance VC-dimension at most $18$ such that the $2$VC-dimension of the $B_\ell$-hypergraph of $G_{n,\ell}$ is at least $n$.
\end{theorem}
\begin{proof}
\begin{figure}
\centering
\includegraphics[scale=0.9]{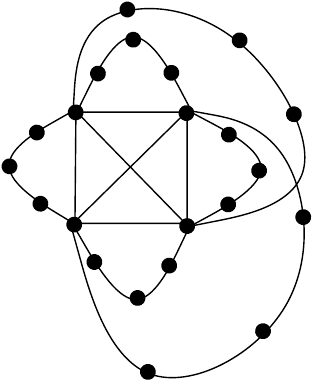}
\caption{The graph $G_{n,\ell}$ of Theorem~\ref{unbounded2VC} with $n=4$ and $\ell=2$. The vertices of the central clique are the vertices of $X$, the others are the vertices of $Y$.}
\label{figcounterexample2vc}
\end{figure}
The following construction is illustrated on Figure~\ref{figcounterexample2vc}. The graph $G_{n,\ell}$ has vertex set $X \cup Y$. The set $X$ contains $n$ vertices denoted by $(x_i)_{1 \leq i \leq n}$ and $Y$ is a set of $(2\ell-1){ n \choose 2}$ vertices denoted by $y_k^{i,j}$ where $1 \leq k \leq 2\ell-1$ and $1 \leq i <j \leq n$. The graph restricted to $X$ is a clique. The graph restricted to $Y$ is a disjoint union of ${n \choose 2}$ induced paths on $2\ell-1$ vertices (whose endpoints will be connected to vertices of $X$). More formally, for every $1 \leq i <j \leq n$ and $k \leq 2\ell-1$, the neighbors of the vertex $y_k^{i,j}$ are the vertices $y_{k-1}^{i,j}$ and $y_{k+1}^{i,j}$ where $y_0^{i,j}$ is $x_i$ and $y_{2\ell}^{i,j}$ is $x_j$. For every $i < j$, the path $x_i, y_1^{i,j}, y_2^{i,j}, \ldots,  y_{2\ell-1}^{i,j}, x_j$ is called the \emph{long path between $x_i$ and $x_j$}. 

The $2$VC-dimension of the $B_\ell$-hypergraph of $G_{n,\ell}$ is at least $n$. Indeed the set $X$ is $2$-shattered since for every $x_i,x_j \in X$, we have $B(y_\ell^{i,j},\ell) \cap X = \{ x_i,x_j \}$. \\
The remaining of the proof consists in showing that the distance VC-dimension of $G_{n,\ell}$ is at most $18$.
Consider an induced subgraph of $G_{n,\ell}$. The remaining vertices of $X$ are in the same connected component since $X$ is a clique. Connected components with no vertices of $X$ form induced paths and then have distance VC-dimension at most two by Theorem~\ref{kdminor}. Thus, by Observation~\ref{obs:connectedgraphs}, Theorem~\ref{unbounded2VC} holds if it holds for the connected component of $X$. 

\begin{claim}\label{no3longpath}
A shattered set of size at least four has at most two vertices on each long path.
\end{claim}
\begin{proof}
Let $z_1, z_2, z_3$ be three vertices which appear in this order on the same long path $P$ and $z_4$ be a vertex which is not between $z_1$ and $z_3$ on $P$. By construction, every path between $z_2$ and $z_4$ intersects either $z_1$ or $z_3$. So no pair $z,p \in V \times \mathbb{N}$ satisfy $B(z,p) \cap X = \{ z_2 , z_4 \}$, \emph{i.e.} $\{ z_1,z_2,z_3,z_4 \}$ is not shattered.
\end{proof}

Let $Z'$ be a shattered set of size at least $19$. By Claim~\ref{no3longpath}, we can extract from $Z'$ a set $Z$ of size $10$ such that vertices of $Z$ are in pairwise distinct long paths. For every vertex $z_i \in Z$, a \emph{nearest neighbor} on $X$ is a vertex $x$ of $X$ such that $d(x,z_i)$ is minimum. Each vertex has at most two nearest neighbors which are the endpoints of the long path containing $z_i$.

First assume $z_1, z_2, z_3$ in $Z$ have a common nearest neighbor $x$, \emph{i.e.} they are on long paths containing $x$ as endpoint. Without loss of generality $d(z_3,X)$ is minimum. Let $z,p$ be such that $\{ z_1, z_2 \} \subseteq  B(z,p)$. Since $z_1$ and $z_2$ are not in the same long path, free to exchange $z_1$ and $z_2$, a minimum $zz_2$-path passes through a vertex $y$ of $X$. If $y=x$, then $B(z,p)$ contains $B(x,d(x,z_2))$ by Lemma~\ref{ballcontain}, and then contains $z_3$ since $d(x,z_2) \geq d(x,z_3)$. Otherwise up to symmetry $y$ is not an endpoint of the long path containing $z_2$. Indeed the second endpoint of the long path containing $z_1$ and the second endpoint of the long path containing $z_2$ are distinct. Otherwise $z_1,z_2$ would be in the same long path since there is a unique long path between every pair of vertices of $X$. Hence a minimum path from $y$ to $z_2$ is at least $d(z_2,X)+1$. In addition a minimum path between $y$ and $z_3$ has length at most $1+d(z_3,X)$. So $d(y,z_2) \geq d(y,z_3)$. So $z_3$ is in $B(z,p)$ and $\{z_1,z_2,z_3\}$ cannot be shattered.

So each vertex of $Z$ has at most two nearest neighbors in $X$ and each vertex of $X$ is the nearest neighbor of at most two vertices of $Z$. Thus every $z \in Z$ share a common nearest neighbor with at most two vertices of $Z$. Since $|Z| \geq 10$, at least four vertices $z_1,z_2,z_3,z_4$ of $Z$ have distinct nearest neighbors. Assume w.l.o.g. that $d(z_4,X)$ is minimum.

Let $z,p \in V \times \mathbb{N}$ be such that $B(z,p)$ contains $z_1,z_2,z_3$. Let $x_1,x_2$ be the endpoints of the long path containing $z$ (if $z \in X$ we consider that $x_1=x_2=z$). Since nearest neighbors of $z_1,z_2,z_3$ are pairwise disjoint, we can assume w.l.o.g. that the nearest neighbors of $z_3$ are distinct from $x_1$ and from $x_2$. So a minimum path from $z$ to $z_3$ passes through $x_1$ or $x_2$ and we have $d(x_1,z_3) \geq d(x_1,z_4)$ and $d(x_2,z_3) \geq d(x_2,z_4)$. By Lemma~\ref{ballcontain}, $B(z,p)$ also contains $z_4$, \emph{i.e.} $Z$ cannot be not shattered. 
\end{proof}
Note that we did not make any attempt to exactly evaluate the distance VC-dimension of the graph $G_{n,\ell}$.

\section{Erd\H{o}s-P\'osa property}\label{VC}
 Recall that $\nu_{\ell}$ and $\tau_{\ell}$ respectively denote the packing number and the transversality of the $B_\ell$-hypergraph of $G$. Chepoi, Estellon and Vax\`es proved in \cite{ChepoiEV07} that there is a constant  $c$ such that for all $\ell$, every planar graph $G$ of diameter $2 \ell$ can be covered by $c$ balls of radius $\ell$. It means that planar graphs of diameter $2 \ell$ satisfy $\tau_{\ell} \leq f(\nu_{\ell})$ since two balls of radius $\ell$ necessarily intersect. 
 They conjectured that there exists a linear function $f$ such that for every $\ell$ and every planar graph we have $\tau_{\ell} \leq f(\nu_{\ell})$. The following result due to Ding, Seymour and Winkler \cite{DingSW94} ensures that a polynomial function $f$ exists for any class of graphs of bounded distance $2$VC-dimension.
\begin{theorem}{(Ding, Seymour, Winkler \cite{DingSW94})}\label{ding}
Each hypergraph of dual $2$VC-dimension $d$ satisfies,
$$ \tau \leq 11 \cdot d^2 \cdot (d+ \nu + 3)\cdot {d+\nu \choose d}^2$$
\end{theorem}

\begin{corollary}\label{erdosposa2VC}
Let $d$ be an integer. For every graph $G \in \mathcal{G}$ and every integer $\ell$, if the distance $2$VC-dimension of $G$ is at most $d$, then
$$ \tau_{\ell} \leq 11 \cdot d^2 \cdot (d+ \nu_{\ell} + 3)\cdot {d+\nu_{\ell} \choose d}^2$$

\end{corollary}
\begin{proof}
Let $G$ be a graph. Observation~\ref{obs:Blisomdual} ensures that the $B_\ell$-hypergraph of $G$ is isomorphic to its dual hypergraph. The $B_\ell$-hypergraph of $G$ is a sub-hypergraph (in the sense of hyperedges) of the $B$-hypergraph of $G$. Hence the dual $2$VC-dimension of the $B_\ell$-hypergraph of $G$ is at most $d$ and then Theorem~\ref{ding} can be applied.
\end{proof}
Theorem~\ref{kdminor},~\ref{rankwidthborne} and Corollary~\ref{erdosposa2VC} ensure that $B_\ell$-hypergraphs of $K_n$-minor free graphs and of bounded rankwidth graphs have the Erd\H{o}s-P\'osa property. Note that the gap function is a polynomial function when the $2$VC-dimension is fixed constant. In particular, Corollary~\ref{erdosposa2VC} implies that every planar graph of diameter $2\ell$ has a dominating set at distance $\ell$ of size $35200$ ($\nu_\ell=1$, $d=4$).
Since Theorem~\ref{unbounded2VC} ensures that there are some graphs with bounded distance VC-dimension and unbounded distance $2$VC-dimension, Corollary~\ref{erdosposa2VC} raises a natural question. Does the same hold for graphs of bounded distance VC-dimension? The remaining of this section is devoted to answering this question.

\begin{theorem}\label{erdpo}
There exists a function $f$ such that, for every $\ell$, every graph of distance VC-dimension $d$ can be covered by $f(\nu_{\ell},d)$ balls of radius $\ell$, \emph{i.e.} $\tau_{\ell} \leq f(\nu_{\ell},d)$.
\end{theorem}
Our proof is based on a result of Matou\v{s}ek linking $(p,q)$-property and Erd\H{o}s-P\'osa property~\cite{Matousek04} (Chepoi, Estellon and Vax\`es use this method in their paper). Nevertheless our proof is more technically involved since we cannot use topological properties as for planar graphs in~\cite{ChepoiEV07}. A hypergraph has the \emph{$(p,q)$-property} if for every set of $p$ hyperedges, $q$ of them have a non-empty intersection, \emph{i.e.} there is a vertex $v$ in at least $q$ of the $p$ hyperedges. The following result, due to Matou\v{s}ek \cite{Matousek04}, generalizes a result of Alon and Kleitman \cite{AlonK92}.

\begin{theorem}{(Matou\v{s}ek \cite{Matousek04})}\label{Matou}
There exists a function $f$ such that every hypergraph $H$ of dual VC-dimension $d$ satisfying the $(p,d+1)$-property satisfies 
$$\tau(H) \leq f(p,d)$$
\end{theorem}
Let $d$ be an integer. Let $G$ be a graph of distance VC-dimension $d$. By Observation~\ref{obs:Blisomdual}, the dual VC-dimension of the $B_\ell$-hypergraph is at most $d$. Hence if there exists a function $p$ such that, for every $\ell$ and every graph $G$ of distance VC-dimension $d$, the $B_\ell$-hypergraph of $G$ satisfies the $(p(\nu_{\ell},d),d+1)$-property, then Theorem~\ref{Matou} will ensures that Theorem~\ref{erdpo} holds. 
So for proving Theorem~\ref{erdpo}, it suffices to show that the size of a set of balls of radius $\ell$ which does not contain $(d+1)$ balls intersecting on a same vertex is bounded by a function of $\nu_\ell$ and $d$. The remaining of this section is devoted to proving this result.

\subsection{A lower bound for the distance VC-dimension of a graph}\label{sec:lowerboud}

Let $A$ and $B$ be two disjoint sets. An \emph{interference matrix $M=(A,B)$} is a matrix 
with $| A |$ rows and $| B |$ columns such that for every $(a,b)\in A \times B$, the \emph{entry} 
$m(a,b)$ is a subset of $(A \cup B) \backslash \{a,b\}$. The \emph{size} of an entry is its number of elements. A 
\emph{$k$-interference matrix $M$} is an interference matrix which entries have size at most $k$. If $A'\subseteq A$
and $B'\subseteq B$, the 
\emph{submatrix $M'$ of $M$ induced by $A' \times B'$} is the matrix restricted to the set of rows $A'$ 
and the set of columns $B'$ which entries are $m'(a',b')=m(a',b') \cap (A' \cup B')$. A $0$-interference 
matrix is called a \emph{proper matrix}. A matrix is \emph{square} if $|A|=|B|$. The \emph{size} of a
square matrix is its number of rows.

\begin{lemma}\label{kpol}
Let $k>0$. A $k$-interference square matrix with no proper submatrix of size $n$ has size less than~$kn^3$.
\end{lemma}
\begin{proof}
Let us show that if $M=(A,B)$ is a $k$-interference matrix with size $m=kn^3+1$, then it 
contains a proper submatrix of size $n$. A triple $(i,j,l) \in A \times B \times (A \cup B)$ 
is a \emph{bad triple} if $l \in m(i,j)$ (and then $l \neq i$ and $l \neq j$). 
A bad triple $(i,j,l)$ is \emph{bad for $(X,Y)$} with $X \subseteq A$, $Y \subseteq B$, $|X| = |Y | = n$
if $i \in A, j \in B$ and $l$ is in $A$ or $B$.

For a given bad triple $(i, j,l)$, let us count the number of pairs $(X, Y )$ where $X \subseteq A$, $Y \subseteq B$, and $|X| = |Y | = n$ containing $(i,j,l)$ as a bad triple. 
Let us consider the case $l \in A$ (the case $l \in B$ is obtained similarly). The number of $X$'s containing both $i$ and $l$ is ${ m-2 \choose n-2 }$ since $i \neq l$. The number of $Y$'s containing $j$ is ${ m-1 \choose n-1 }$.
Since $M$ is a $k$-interference matrix, the total number of bad triples is at most $k\cdot m^2$. Thus the total number of pairs $X, Y$ with $X \subseteq A$, $Y \subseteq B$, $|X| = |Y | = n$ is ${ n \choose m}^2$. So if the number of such pairs is larger than the number of pairs containing a bad triple, the conclusion holds. In other words, if
\[ { m-2 \choose n-2 } \cdot { m-1 \choose n-1 } \cdot k m^2 < { n \choose m}^2\]
there is a pair $(X,Y)$ with $X \subseteq A$, $Y \subseteq B$, $|X| = |Y | = n$ which does not contain a bad triple. This latter inequality is equivalent with $kn^2\cdot(n - 1) < m - 1$.
%
\end{proof}

Given a path $P$ from $x$ to $y$ and a path $Q$ from $y$ to $z$, the \emph{concatenation} of $P$ and $Q$ denoted by $PQ$ is the walk consisting on the edges of $P$ followed by the edges of $Q$. The length of a path $P$ is denoted by $|P|$. Let $G=(V,E)$ be a graph and $\prec_l$ be a total order on $E$. We extend $\prec_l$
on paths, for any paths $P_1$ and $P_2$ as follows :

\begin{itemize}
\item If $P_1$ has no edges, then $P_1 \prec_l P_2$.
\item If  $P_1=P_1'.e_1$ and $P_2=P_2'.e_1$, where $e_1$ is the last edge of $P_1$ and $P_2$, 
then $P_1 \prec_l P_2$ if and only if $P_1' \prec_l P_2'$.
\item If  $P_1=P_1'.e_1$ and $P_2=P_2'.e_2$, where $e_1\neq e_2$,
then $P_1 \prec_l P_2$ if and only if $e_1 \prec_l e_2$.
\end{itemize}

The order $\prec_l$ is called the \emph{lexicographic order} (note nevertheless that paths are compared from their end to their beginning).
The \emph{minimum path} from $x$ to $z$, 
also called \emph{the $xz$-path} and denoted by $P_{xz}$, is the path of minimum length with minimum lexicographic 
order from $x$ to $z$. Observe that two minimum paths going to the same vertex $z$ and passing through 
the same vertex $u$ coincide between $u$ and $z$. We note $u \trianglelefteq_{xz} v$ if $u$ appears before $v$ on the $xz$-path.
Given a path from $a$ to $b$ passing through $c$, the \emph{suffix path on $c$} (resp. \emph{prefix path on $c$}) is the $cb$-subpath (resp. $ac$-subpath) of the $ab$-path. Note that every 
suffix of a minimum path is a minimum path. Given two sets $X$ and $Z$, the \emph{$XZ$-paths} are the $xz$-paths for all $x,z \in X \times Z$.

Let $x_1, x_2$ and $z$ be three vertices. Two distinct edges $v_1u_2$ and $u_1v_2$ form a \emph{cross} between the $x_1z$-path and the $x_2z$-path if for $i \in \{1,2\}$, $u_i \trianglelefteq_{x_iz} v_i$ (see Figure \ref{crossesproof}).

\begin{figure}
\center
\includegraphics[scale=0.65]{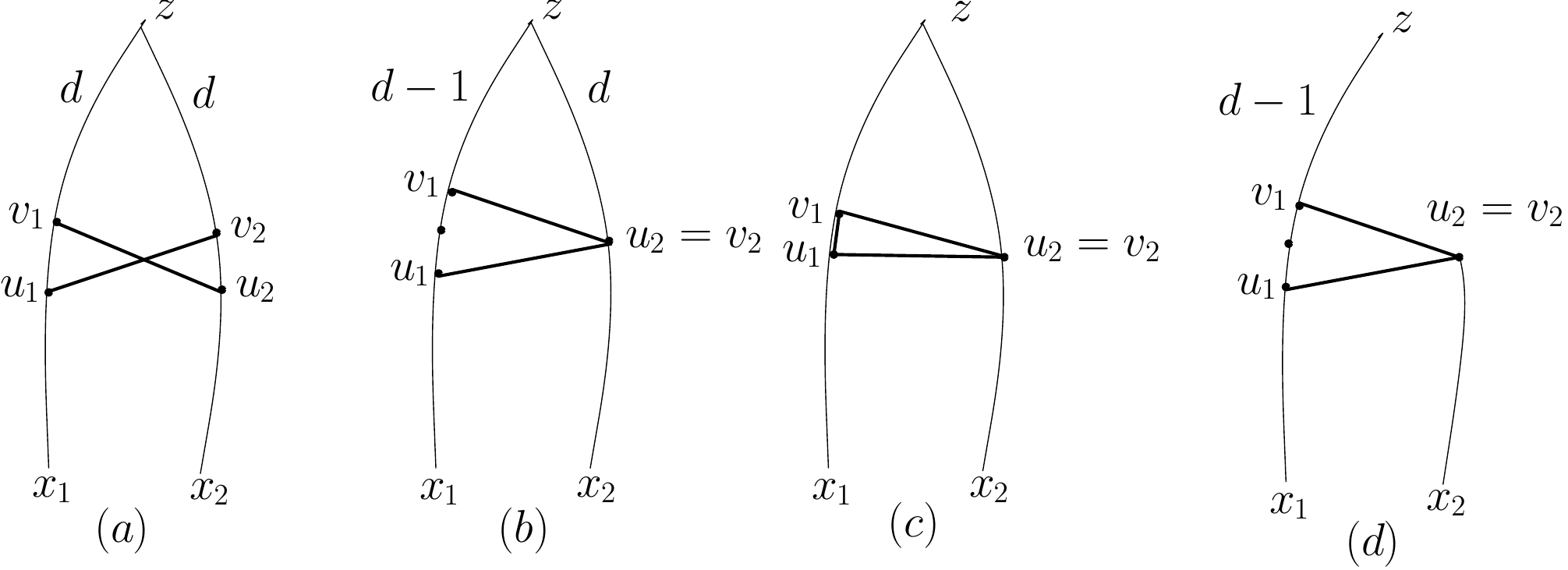}
\caption{$4$ types of crosses. Lemma~\ref{nox} ensures that, up to symmetry, only (c) and (d) are authorized. The thick chords are edges of the graph. Thin chords represent paths. Distances are denoted by $d$ or $d-1$. In the case of Figure~\ref{crossesproof}(d), the path $Q_{v_2}$ is $v_2v_1Q_{v_1}$.}
\label{crossesproof}
\end{figure}

\begin{lemma}\label{nox}
Let $x_1,x_2,z$ be three vertices. If the edges $u_1v_2$ and $v_1u_2$ form a cross between the $x_1z$-path and the $x_2z$-path, then free to exchange $x_1$ and $x_2$ we have:
\begin{itemize}
 \item either $u_2=v_2$ and $u_1v_1$ is an edge. 
 \item Or $u_2=v_2$ and the $v_2z$-path is the edge $v_2v_1$ concatenated with the $v_1z$-path.
\end{itemize}
In other words, only cases (c) and (d) of Figure~\ref{crossesproof} can occur.
\end{lemma}
\begin{proof}
For $i \in \{1,2\}$, we denote by $Q_{u_i}$ (resp. $Q_{v_i}$) the suffix of the $x_iz$-path on $u_i$ (resp. $v_i$). Since suffixes of minimum paths are minimum paths, these four paths are minimum paths. We prove that if a cross does not satisfy the condition of Lemma~\ref{nox}, then one of these paths is not minimum.

A \emph{real cross} is a cross for which $u_1 \neq v_1$ and $u_2 \neq v_2$ (Figure~\ref{crossesproof}(a)). 
A \emph{degenerated cross} is a cross for which, up to symmetry, $u_2=v_2$ and $Q_{v_2} \neq v_2v_1.Q_{v_1}$ (Figure~\ref{crossesproof}(b)). \\
A real cross satisfies $|Q_{v_1}|=|Q_{v_2}|$. Indeed if $|Q_{v_1}| < |Q_{v_2}|$ then $u_2v_1.Q_{v_1}$ has length at most $|Q_{v_2}|$. This path is strictly shorter than $Q_{u_2}$ (since $u_2 \neq v_2$, indeed the cross is a real cross), contradicting the minimality of $Q_{u_2}$. So $|Q_{v_1}|=|Q_{v_2}|$. Free to exchange $x_1$ and $x_2$, we have $Q_{v_1} \prec_l Q_{v_2}$. So $u_2v_1.Q_{v_1} \prec_l Q_{u_2}$ (recall that we first compare the last edge) and $|u_2v_1.Q_{v_1}| \leq |Q_{u_2}|$. So $Q_{u_2}$ is not minimum, a contradiction. Hence there is no real cross. \\
Consider a degenerated cross such that $u_1v_1 \notin E$. In particular $u_1$ and $v_1$ are at distance $2$. So we have $|Q_{v_1}| < |Q_{v_2}|$ otherwise $u_1v_2.Q_{v_2}$ would be strictly shorter than $Q_{u_1}$, a contradiction. In addition, $|Q_{v_2}|$ and $|Q_{v_1}|$  differ by at most one since $v_1v_2$ is an edge. So $|Q_{v_1}| +1 = |Q_{v_2}|$.
Assume now that we are not in the case of Figure~\ref{crossesproof}(d), in other words, $Q_{v_2} \neq v_2v_1Q_{v_1}$. If $Q_{v_2} \prec_l Q_{v_1}$ then $u_1v_2Q_{v_2}$ is not longer than $Q_{u_1}$ (since $u_1$ and $v_1$ are at distance $2$) and has a smaller lexicographic order, a contradiction with the minimality of $Q_{u_1}$. 
If $Q_{v_1} \prec_l Q_{v_2}$ then $v_2v_1.Q_{v_1}$ is not longer and has a smaller lexicographic order, a contradiction with the minimality of $Q_{v_2}$. So either the degenerated cross satisfies $u_1v_1 \in E$ or $Q_{v_2}=v_2v_1.Q_{v_1}$.
\end{proof}

\begin{figure}
\center
\includegraphics[scale=0.75]{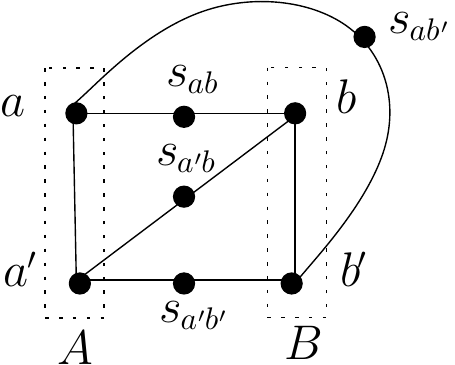}
\caption{The sets  $S_{ab}=\{s_{ab}\}$  are $2$-disconnecting for $A,B$.}
\label{fig:ldisconnecting}
\end{figure}

Let $\ell$ be an integer and $A,B$ be two disjoint subsets of vertices. To every pair $(a,b) \in A \times B$, 
we associate a set of vertices $S_{a,b}$ which is disjoint from $A \cup B$. 
We say that the set of subsets ${\mathcal S}=\{(S_{a,b})_{(a,b) \in A \times B}\}$ is \emph{$\ell$-disconnecting} if
for every subset $C$ of ${\mathcal S}$ and every pair $(a,b)$, we have $d(a,b) > \ell$ in $G\setminus \bigcup C$ if 
and only if $S_{a,b} \in C$. 
If such a family of sets exists, 
then $A,B$ are said to be \emph{$\ell$-disconnectable}. Another way of defining $\ell$-disconnecting families
would be to say that $d(a,b) > \ell$ in $G\setminus S_{a,b}$ and $d(a,b) \leq \ell$ in $G\setminus \bigcup ({\mathcal S}\setminus S_{a,b})$,
or roughly speaking that $S_{a,b}$ is the only set whose deletion can increase $d(a,b)$ above $\ell$.
In Figure~\ref{fig:ldisconnecting}, the sets $A,B$ are $2$-disconnectable. Indeed the deletion of any vertex $s_{ab}$ eliminates all the paths of length at most $2$ from $a$ to $b$. Note nevertheless that $a$ and $b$ are still in the same connected component after this operation.

\begin{theorem}\label{bigvc}
Let $G=(V,E)$ be a graph and $\ell$ be an integer. If there exist two subsets $A,B$ of $V$ with $|B|=2^{|A|}$ which 
are $\ell$-disconnectable, then the distance VC-dimension of $G$ is at least $|A|$.
\end{theorem}
\begin{proof}
Let us prove that the set $A$ can be shattered in the $B_\ell$-hypergraph of an induced subgraph of $G$. Associate in a
one to one way every vertex $b$ of $B$ to a subset $A_b$ of $A$. 
Since $A,B$ are $\ell$-disconnectable, there exists a family ${\mathcal S}$ of subsets which is $\ell$-disconnecting for $A,B$.
Let $C$ be the collection of ${\mathcal S}$ consisting of all the sets $S_{a,b}$ such that $a \in A_b$. Since ${\mathcal S}$ 
is $\ell$-disconnecting, $B(b,\ell) \cap A = A \setminus A_b$ in $G\setminus C$, for all $b\in B$ (the deletion of $S_{a,b}$ eliminates the paths of length at most $\ell$ between $a$ and $b$). Hence the set $A$ is shattered 
by balls of radius $\ell$ in $G\setminus C$. Therefore the distance VC-dimension of $G$ is at least $|A|$.
\end{proof}


\subsection{Sparse sets}\label{secsparse}
Let $G$ be a graph of distance VC-dimension $d$ and $q, \ell$ be two integers. Most of the following definitions depend on $\ell$. Nevertheless, in order to avoid heavy notations, this dependence will be implicit in the terminology. A set of balls of radius $\ell$ is \emph{$q$-sparse} if no vertex of the graph is in more than $q$ balls of the set. Note that a subset of a $q$-sparse set is still $q$-sparse. By abuse of notation, a set $X$ of vertices is called \emph{$q$-sparse} if the set of balls of radius $\ell$ centered in $X$ is $q$-sparse. 

Assume that the $B_\ell$-hypergraph of a graph $G$ does not satisfy the $(p,d+1)$-property. Then there exist $p$ balls of radius $\ell$ such that no vertex is in at least $(d+1)$ of these $p$ balls, \emph{i.e.} there is a $d$-sparse set of size $p$. In other words, a $d$-sparse set of size $p$ is a certificate that the $(p,d+1)$-property does not hold. In order to prove Theorem~\ref{erdpo}, we just have to show that $p$ can be bounded by a function of $d$ and $\nu_l$. The remaining of this section is devoted to show that there exists a function $f$ such that the size of a $d$-sparse set is at most $f(d,\nu_l)$. 

A set $X$ of vertices is \emph{$d$-localized} if the vertices of $X$ are pairwise at distance at least $\ell+1$ and at most $2\ell-2^{d+2}-3$. A $d$-localized set is defined only if this value is positive. A pair $A,B$ of disjoint sets of vertices is \emph{$q$-sparse} if $A \cup B$ is. A disjoint pair $A,B$ of vertices is \emph{$d$-localized} if the vertices of $A \cup B$ are pairwise at distance at least $\ell+1$, and if for every $a,b \in A \times B$, $d(a,b) \leq 2\ell-2^{d+2}-3$. A subpair of a $d$-localized pair is $d$-localized. The \emph{size} of a pair $A,B$ is $\min(|A|,|B|)$.

\begin{theorem}\label{ramsey}(Ramsey)
There exists a function $r_k$ such that every complete edge-colored graph $G$ with $k$ colors 
with no monochromatic clique of size $n$ has at most $r_k(n)$ vertices.
\end{theorem}
All along the paper, logarithms are in base $2$.
\begin{theorem}\label{memelongueur}
Let $G$ be a graph and $X$ be a subset of vertices pairwise at distance exactly $r$. Assume also 
that no vertex of $G$ belongs to $q$ balls of radius $\lceil r/2 \rceil$ with centers in $X$. Then 
the distance VC-dimension of $G$ is at least $(\log |X| -\log 2q)/3$.
\end{theorem}
\begin{proof}
Let $r'$ be equal to $\lceil r/2 \rceil$. Free to remove one vertex from $X$, we can assume 
that $X$ is even, and we consider a partition $A,B$ of $X$ with $|A| = |B|$. For every pair 
$(a,b)\in A \times B$, we denote the minimum $ab$-path by $P_{ab}$. By abuse of notation, we 
still denote by $G$ the restriction of $G$ to the vertices of the union of the paths $P_{ab}$ for all $a\in A$ and 
$b \in B$. Observe that we preserve the hypothesis
of Theorem~\ref{memelongueur} apart from the fact that the distance between vertices inside $A$ (resp. inside $B$)
may have increased above $r$. Let $a\in A$ and $b \in B$. Let $y$ be a vertex of $X$ distinct from $a$ and $b$. 
If $B(y,r')\cap P_{ab} \neq \varnothing$, then denote by $x$ a vertex in this set.
We have $d(a,x) \geq \lfloor r/2 \rfloor$ since $d(a,y)\geq r$ and $d(y,x)\leq\lceil r/2 \rceil$. 
By symmetry, we also have $d(b,x) \geq \lfloor r/2 \rfloor$. Hence $x$ is a 
\emph{midvertex} of $P_{ab}$, \emph{i.e.} a vertex of $P_{ab}$ at distance  $\lfloor r/2 \rfloor$ or $\lceil r/2 \rceil$ 
from $a$ (and thus also from $b$). Recall that a midvertex $x$ of $P_{ab}$ belongs to at most 
$q-1$ balls of radius $r'$ (including $B(a,r')$ and $B(b,r')$).

Consider the interference matrix $M=(A,B)$ where 
$m(a,b)=\{ y\in (A \cup B) \backslash \{ a,b \} | B(y,r') \cap P_{ab} \neq \varnothing \}$. Since
$P_{ab}$ has at most two midvertices and each of these belongs to at most $q-3$
balls  $B(y,r')$ with $y$ different from $a$ and $b$, the matrix $M$ is a $(2q-6)$-interference 
matrix. To avoid tedious calculations and free to increase the interference value, we only 
assume that $M$ is a $2q$-interference 
matrix (with $2q\geq 1$). By Lemma~\ref{kpol}, there is a proper submatrix $M'$ 
of size $N=(|X|/2q)^{1/3}$. Let us denote by $A'$ the set of rows and $B'$ the set of columns of the extracted 
matrix. Let us still denote by $G$ the restriction of the graph to the vertices of the paths $(P_{ab})_{(a,b) \in A' \times B'}$. 

Let $a,a'\in A'$ and $b'\in B'$. The key-observation is that if $B(a,r')$ intersects $P_{a'b'}$, then $a=a'$. Indeed, by definition of $M$, we have $a \in m(a',b)$,
contradicting the fact that $M'$ is a proper submatrix.

Let $M_{ab}$ be the set of midvertices of $P_{ab}$, where $a,b \in A' \times B'$. We claim that 
$M_{ab}$ is disjoint from $P_{a'b'}$, whenever $P_{a'b'}\neq P_{ab}$. Indeed if 
$x\in M_{ab}\cap P_{a'b'}$, we have in particular both $d(a,x)\leq r'$ and $d(b,x)\leq r'$,
and thus by the key-observation $a=a'$ and $b=b'$. In other words, deleting 
$M_{ab}$ never affects $P_{a'b'}$, whenever $P_{a'b'}\neq P_{ab}$.

Another crucial remark is that every path $P$ of length $r$ from $a$ to $b$ intersects $M_{ab}$. Indeed,
let $x$ be a vertex of $P$ with both $d(a,x)\leq r'$ and $d(b,x)\leq r'$. Since $x$ is in $G$, it belongs to some path 
$P_{a'b'}$. By the key-observation, we both have $a'=a$ and $b'=b$, hence $x\in M_{ab}$.

To conclude, observe that the deletion of $M_{ab}$ ensures that the distance $d(a,b)$ is more than $r$
whereas deleting the union of all $M_{a'b'}$ different from $M_{ab}$ does not affect
$d(a,b)$ which is still equal to $r$.
Consequently, the sets $(M_{ab})_{(a,b) \in A' \times B'}$ are $r$-disconnecting for $A',B'$. 
Hence, by Theorem~\ref{bigvc}, the distance VC-dimension of $G$ is at least $\log(N)=(\log |X| -\log 2q)/3$.
\end{proof}

\begin{lemma}\label{chemincourtlong}
Let $G$ be a graph of distance VC-dimension at most $d$. There exists a function $f$ such that:
\begin{enumerate}[(a)]
 \item Either $G$ contains a $d$-localized set of size $p$ which is $d$-sparse,
 \item Or the $(f(\nu_\ell,d,p),d+1)$-property holds.
\end{enumerate}
\end{lemma}
\begin{proof}
Let $D=2^{d+2}+2$ and $N=\max(p,\nu_l+1,2^{3d+3+\log(4d+2)})$. Let $f$ be a function such that $f(\nu_\ell,d,p) \geq r_{D+4}(N)+1$. 
Let us show that function $f$ satisfies Lemma~\ref{chemincourtlong}. Assume that point (b) does not hold, 
\emph{i.e.} the $(f(\nu_\ell,d,p),d+1)$-property does not hold. So there is a subset $X$ of vertices of size $r_{D+4}(N)+1$ such that the set $X$ is $d$-sparse.
Let us show that point (a) holds.

Consider the complete $(D+4)$-edge-colored graph $G'$ with vertex set $X$ such that, for every $x,y \in X$, $xy$ has color:
\begin{itemize}
\item $c$ with $0 \leq c \leq D$ if $d(x,y)=2\ell-c$,
\item $D+1$ if $d(x,y) \leq \ell$,
\item $D+2$ if $d(x,y)>2\ell$,
\item $D+3$ otherwise.
\end{itemize}
Theorem~\ref{ramsey} ensures that there is a monochromatic clique $K$ of size $N$.
Let $K'$ be a clique of color $D+1$ and  $x \in K'$. Then $K' \subseteq B(x,\ell) \cap X$. Thus the size of $K'$ is at most $d$ since $X$ is $d$-sparse.
At most $\nu_{\ell}$ balls of radius $\ell$ centered in $X$ are vertex disjoint by definition of the packing number. Thus the size of a clique of color $D+2$ is at most $\nu_{\ell}<N$.
Since $X$ is $d$-sparse, then $K$ also is. Then, for every $0 \leq c \leq D$, no vertex of $G$ belongs to $(d+1)$ balls of radius $\lceil (2\ell -c)/2 \rceil \leq \ell$ centered in $X$. 
Therefore the color of $K$ cannot be in $0 \leq r \leq D$. Otherwise Theorem~\ref{memelongueur} would ensure that the distance VC-dimension of $G$ is at least $\log(N)/3 -\log(4d+2)/3 \geq d+1$.
So the clique $K$ of size $N \geq p$ has color $D+3$. A clique of color $D+3$ defines a $d$-localized set. Moreover $K$ is $d$-sparse since $X$ is.
Thus $K$ satisfies (a).
\end{proof}

The vertices of a $d$-localized set have to be pairwise at distance at least $d+1$ and at most $2\ell - 2^{d+2}-3$. The edge-colored graph of Lemma~\ref{chemincourtlong} was constructed in order to ensure this property. 


\subsection{Localized and independent pairs}\label{secconcentrated}

In this section we introduce a notion of independence for every pair of vertices. We first give some properties of independent pairs and we will finally show that any large enough $d$-sparse and $d$-localized pair contains a large enough independent subpair.

\begin{figure}
\centering
\includegraphics{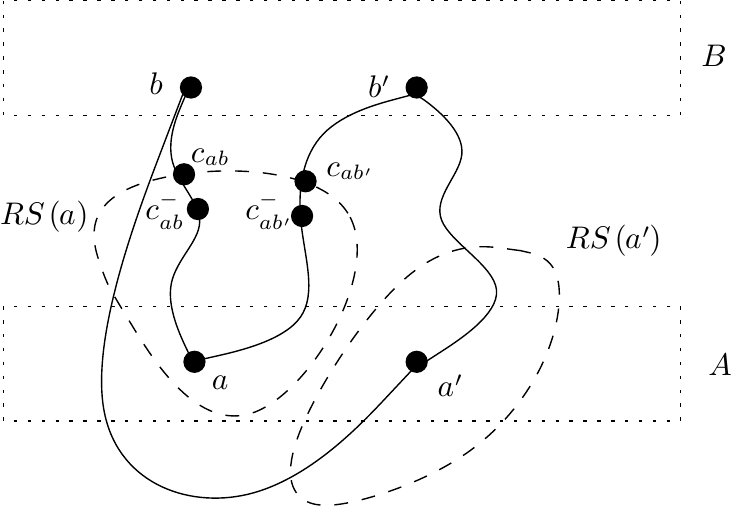}
\caption[scale=0.9]{Minimum paths with root sections (dashed parts), critical vertices and pre-critical vertices.}
\label{figrootsection}
\end{figure}

Let $A,B$ be a $d$-localized pair. In the following we consider the restriction of the graph to $\cup_{a \in A,b \in B} P_{ab}$. Recall that $P_{ab}$ is the minimum path with minimum lexicographic order from $a$ to $b$, also called the $ab$-path. Note that the sets $A$ and $B$ are not treated symmetrically since we only consider the minimum paths from $A$ to $B$. Let $a,a' \in A$ and $b \in B$. Note that, since $d(a,a')>\ell$, $d(a',b)> \ell$ and $d(a,b) <2\ell$, the vertex $a'$ does not belong to $P_{ab}$.

For every pair $a,b \in A \times B$, the \emph{critical vertex} $c_{ab}$ (resp. $c_{ba}$) is the vertex of $P_{ab}$ at distance $\ell-3$ from $a$ (resp. $b$) and the \emph{pre-critical vertex $c^{-}_{ab}$} is the vertex of $P_{ab}$ at distance $\ell-4$ from $a$ (see Figure~\ref{figrootsection}). Such vertices exist since $d(a,b) > \ell$. Moreover $c_{ab}$ and $c^{-}_{ab}$ are adjacent. Note that both $c_{ab}$ and $c_{ba}$ are vertices of $P_{ab}$. In the following, we mostly need the vertex $c_{ba}$ in order to ensure some distance properties (and then we do not use the minimality of the lexicographic order for these vertices). On the contrary, the vertex $c_{ab}$ will be used for both distance and lexicographic arguments. The \emph{root section of $a \in A$} (resp. $b \in B$), denoted by $RS(a)$ (resp. $RS(b)$), is the set of vertices of the $ac_{ab}$-subpaths (resp. $c_{ba}b$-subpaths) of $P_{ab}$ for all $b \in B$ (resp. $a \in A$). We denote by $RS(A)$ the set $\cup_{a \in A} RS(a)$.

Since $d(a,b) \leq 2\ell-7$, the vertex $c_{ba}$ precedes the vertex $c_{ab}$ on the path $P_{ab}$. In particular we have $P_{ab} \subseteq RS(a) \cup RS(b)$, hence every vertex of $G$ belongs to some root section. In fact, we have the slightly stronger following observation:

\begin{obs}\label{obscritical}
For every $a,b$ in $A \times B$, the critical vertex $c_{ab}$ and the pre-critical vertex $c^{-}_{ab}$ are in $RS(b)$.
\end{obs}

A $d$-localized pair $A,B$ is \emph{independent}, if for every $a,b \in A \times B$, the ball $B(c_{ab},\ell)$ intersects $A \cup B$ on $\{a,b\}$ and $B(c_{ba},\ell)\cap (A \cup B) =\{a,b\}$. A subpair of an independent pair is still independent. In addition, $A,B$ is still independent in the graph induced by the vertices of the $AB$-paths. 

\begin{lemma}\label{depol}
The size of a $d$-sparse and $d$-localized pair with no independent subpair of size $p$ is at most $2d \cdot p^3$.
\end{lemma}
\begin{proof}
Let $A,B$ be a $d$-sparse and $d$-localized pair of size $2d \cdot p^3+1$. For every vertex~$u$, $I(u)$ denotes $B(u,\ell) \cap (A \cup B)$. Since $A \cap B = \varnothing$, the matrix $M=(A,B)$ where $m(a,b)=(I(c_{ab}) \cup I(c_{ba})) \backslash \{a,b\}$, is a well-defined interference matrix.
The pair $A,B$ is $d$-sparse, then $|I(u)| \leq d$ for every vertex $u$. Thus $M$ is a $2d$-interference matrix. 

By Lemma~\ref{kpol}, $M$ has a proper submatrix $(A',B')$ of size $p$. Thus for every $a',b' \in A' \times B'$, $B(c_{a'b'},\ell) \cap (A' \cup B') = \{ a',b'\}$ and the same holds for $c_{b'a'}$, \emph{i.e.} $A',B'$ is independent.
\end{proof}

\begin{lemma}\label{nonintche}
Let $A,B$ be an independent pair.
\begin{enumerate}[(a)]
\item Every pair of vertices of endpoints disjoint $AB$-paths are at distance at least $4$. 
\item For every pair $a,a'$ in $A$ (resp. $b,b'$ in $B$), $d(RS(a),RS(a')) \geq 4$ (resp. $d(RS(b),RS(b') \geq 4$).
\end{enumerate}
\end{lemma}
\begin{proof}
Let us first prove (b). We prove it for vertices of $A$, the case of vertices of $B$ will handle symmetrically (indeed the proof rely on distance arguments and not lexicographic ones).
Let $a\neq a'$ with $u \in RS(a)$ and $u' \in RS(a')$. There exists $b$ and $b'$ in $B$ such that $u$ is in the prefix path on $c_{ab}$ of the $ab$-path and $u'$ is in the prefix path on $c_{a'b'}$ of the $a'b'$-path. Free to exchange $a$ and $a'$, $d(a,u) \leq d(a',u')$. Since $d(a',c_{a'b'})=\ell-3$, we have $d(a,u)+d(u',c_{a'b'}) \leq d(a',u')+d(u',c_{a'b'})=\ell-3$. Since $A,B$ is independent, $d(a,c_{a'b'}) >\ell$, we have $\ell < d(a,u)+d(u,u')+d(u',c_{a'b'}) \leq \ell-3 +d(u,u')$, and then $d(u,u') \geq 4$. So (b) holds.

Let $u$ be a vertex of the $ab$-path, and $u'$ be a vertex of the $a'b'$-path such that $a \neq a'$ and $b \neq b'$. By part (b) of Lemma~\ref{nonintche} we may assume without loss of generality that $u \in RS(a)$ and $u' \in RS(b')$. In addition, we can assume that $d(a,u) \leq d(b',u')$. So $d(a,u)+d(u',c_{b'a'}) \leq d(b',u')+d(u',c_{b'a'})=\ell-3$. So $\ell < d(a,c_{a'b'}) \leq d(a,u)+d(u,u')+d(u',c_{b'a'}) \leq \ell-3 +d(u,u')$. Hence $d(u,u') \geq 4$.
\end{proof}

An edge \emph{leaves} a set $S$ if exactly one of its endpoints is in $S$. 

\begin{obs}\label{obscritical2}
Let $A,B$ be an independent pair and $a \in A$. For all $b \neq b'$, we have $c^{-}_{ab} \neq c^{-}_{ab'}$ (and then $c_{ab} \neq c_{ab'}$). Moreover the edges of the $aB$-paths leaving $RS(a)$ form an induced matching.
\end{obs}
Recall that, by lexicographic minimality, when two $aB$-paths separate, they never meet again, so if  $c^{-}_{ab} \neq c^{-}_{ab'}$, we immediately have $c_{ab} \neq c_{ab'}$.
\begin{proof}
Observation~\ref{obscritical} ensures that $c^{-}_{ab} \in RS(b)$ and $c^{-}_{ab'} \in RS(b')$. So Lemma~\ref{nonintche}(b) ensures that $c^{-}_{ab} \neq c^{-}_{ab'}$. The lexicographic minimality ensures that edges of $aB$-paths leaving $RS(a)$ are vertex disjoint, \emph{i.e.} they form a (non necessarily induced) matching. By Observation~\ref{obscritical}, the edge of $P_{ab}$ leaving $RS(a)$ is an edge with both endpoints in $RS(b)$. Thus Lemma~\ref{nonintche}(b) ensures that the matching is induced.
\end{proof}

\subsection{Escape property}\label{secescape}

\begin{figure}
\center\includegraphics[scale=0.85]{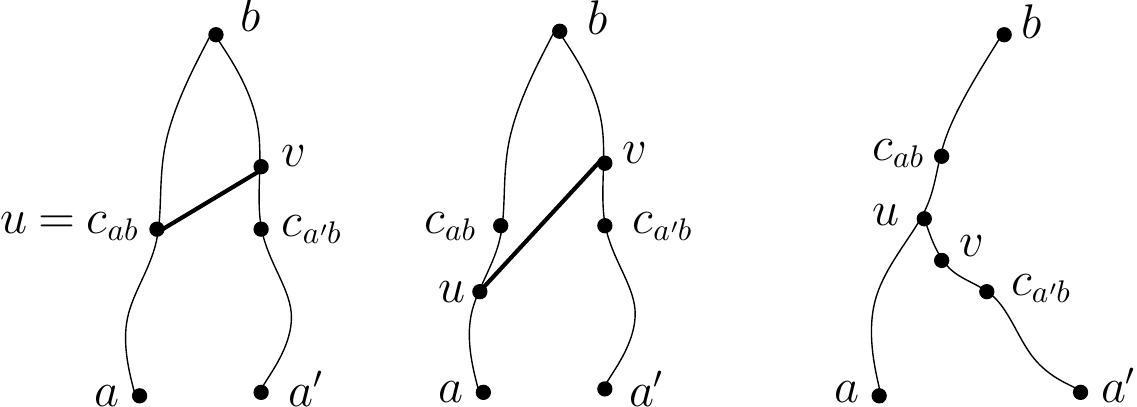}
\caption{Examples of escapes. The right one is an edge of the $a'b$-path.}
\label{dessinshort}
\end{figure}
Let $A,B$ be an independent pair. In the following we consider the restriction of the graph to the vertices of the $AB$-paths.
Let $a$ in $A$. An \emph{escape $uv$ from $a$} is an edge leaving $RS(a)$ such that $uv$ is not an edge of any $P_{ab}$ for $b \in B$.
By convention, when $uv$ is an escape from $a$, we still denote by $u$ the vertex in $RS(a)$ and by $v$ the vertex which is not in $RS(a)$. The vertex $u$ is called the \emph{beginning} of the escape and $v$ the \emph{end} of the escape. 

Let $uv$ be an escape from $a$. Since $u \in RS(a)$, there exists $b \in B$ such that the vertex $u$ is in $P_{ab}$. Since we have considered the restriction of the graph to the vertices of the $AB$-paths, the vertex $v$ is in the path $P_{a'b'}$ for $a' \in A$ and $b' \in B$. Lemma~\ref{nonintche}(a) ensures that either $a=a'$ or $b=b'$. If $a=a'$ then $d(a,v)>\ell-3$ (otherwise $v$ would be in $RS(a)$). So we have $d(a,u)=\ell-3$ since $u \in RS(a)$ and $uv$ is an edge. Though the induced matching property of Observation~\ref{obscritical2} ensures that there is no edge between $c_{ab}$ and $v$ (otherwise the edges leaving $P_{ab}$ and $P_{ab'}$ do not form an induced matching). So $a \neq a'$, \emph{i.e.} $b=b'$.  Thus every escape $uv$ an \emph{escape from $a$ to $a'$ for $b$}. In Figure~\ref{dessinshort}, the edges $uv$ are escapes from $a$ to $a'$ for $b$.
An escape can be an edge of a minimum path (see the rightmost example of Figure~\ref{dessinshort}).

A \emph{deep escape} is an escape such that $u$ is neither a critical vertex nor a pre-critical vertex. Let us define two graphs: the \emph{escape graph of $b$} (resp. \emph{deep escape graph of $b$}) is a directed graph with vertex set $A$ where $aa'$ is an arc if there is an escape (resp. a deep escape) from $a$ to $a'$ for $b$. In Figure~\ref{dessinshort}, the leftmost escape is not a deep escape since $u=c_{ab}$.

If a vertex $x$ which is not in $RS(A)$ has a neighbor in $RS(a)$, $a$ is called \emph{an origin root section} on $x$. Lemma~\ref{nonintche}(b) ensures that every vertex has at most one origin root section (otherwise two root sections would be at distance $2$). Note that if $uv$ is an escape from $a$, then $a$ is the origin root section of $v$. 

Let us informally explain why we introduce escapes. As long as a path from $a$ to $B$ follow edges of $aB$-paths, then we can understand the structure of the path. In particular, if such a path passes through a critical (or pre-critical vertex) we can ``evaluate'' its length using the fact that $d(a,c_{ab})=\ell-3$. If a path uses an escape, it can ``escape'' $RS(a)$ without passing through such a vertex, which implies that the length of the path is somehow harder to evaluate. Let us first show that the structure of the (deep) escape graph can be constraint.

\begin{lemma}\label{acyclic}
Let $A,B$ be an independent pair. For every $b \in B$, the escape graph of $b$ has no circuit.
\end{lemma}
\begin{proof}
Assume that there is a circuit $a_0,a_1,\ldots,a_k,a_0$. In the following indices have to be understood modulo $k+1$. For every $i$, let $u_iv_i$ be an escape from $a_i$ to $a_{i+1}$ for $b$. Since $u_{i} \in RS(a_{i})$ and $u_{i+1} \in RS(a_{i+1})$, Lemma~\ref{nonintche}(b) ensures that $d(u_{i},u_{i+1}) \geq 4$, then $d(v_{i},u_{i+1}) \geq 3$. Hence $d(b,u_{i})\leq d(b,v_{i})+1 < d(b,v_i)+d(v_i,u_{i+1}) = d(b,u_{i+1})$. The first inequality comes from the fact that $u_iv_i$ is an edge and the last equality comes from the fact that the path is a minimum path. A propagation of these inequalities along the arcs of the circuit leads to $d(b,u_0)<d(b,u_0)$, a contradiction.
\end{proof}

The deep escape graph of $b$ is a subgraph, in the sense of arcs, of the escape graph of $b$. Thus the deep escape graph of $b$ has no circuit. For every $b$, the \emph{order inherited from $b$} is a partial order on $A$ such that $a < a'$ if and only if there is an escape from $a$ to $a'$ for $b$. An independent pair $A,B$ has the \emph{escape property} if for every $b \in B$, the deep escape graph of $b$ is a transitive tournament.

\begin{lemma}\label{sgclique}
The size of an independent pair with no subpair of size $2^{d+1}$ satisfying the escape property is at most $r_{2^{2^{d+2}}}(2^{d+1})$.
\end{lemma}

\begin{proof}
Let $(A,B)$ be an independent pair of size $r_{2^{2^{d+2}}}(2^{d+1})+1$.
\begin{claim}\label{cliqueorstable}
$A,B$ has a subpair $X,Z$ of size $2^{d+1}$ such that:
\begin{enumerate}[(1)]
\item either the pair $X,Z$ does not contain a deep escape,
\item or the pair $X,Z$ satisfies the escape property.
\end{enumerate}
\end{claim}
\begin{proof}
Let $B'=\{b_1,\ldots,b_{2^{d+2}}\}$ be a subset of $B$ of size $2^{d+2}$. Consider the complete edge-colored graph $G'$ on vertex set $A$. The colors are binary integers of $2^{d+2}$ digits. The $i$-th digit of the color of $aa'$ is $1$ if there is a deep escape from $a$ to $a'$ (or from $a'$ to $a$) for $b_i$ and $0$ otherwise. 
Theorem~\ref{ramsey} ensures that $G'$ contains a monochromatic clique $X$ of size $2^{d+1}$. Let us denote by $c$ the color of the edges of $G'[X]$. 
At least $2^{d+1}$ digits of $c$ are equal. Denote by $Z$ the subset of $B'$ corresponding to these digits. If the digits equal $0$ then (1) holds, otherwise (2) holds.
\end{proof}

Let us prove by contradiction that Claim~\ref{cliqueorstable}(1) cannot hold. Let $X,Z$ be an independent pair with no deep escape. Consider the restriction of the graph to $\bigcup_{x,z} P_{xz}$.  For every $x,z$, the \emph{private part of $xz$}, denoted by $PP(x,z)$, is the set of vertices which belong to $P_{xz}$ and which do not belong to any other path in $P_{XZ}$.

\begin{claim}\label{relsect}
 $PP(x,z)$ separates $x$ from $c_{xz}$ and from $c^{-}_{xz}$ in the graph induced by $RS(x)$.
\end{claim}
\begin{proof}
Let $P$ be a path from $x$ to $c_{xz}$ in $RS(x)$ and let $u$ be the last vertex of $P$ which is on $P_{xz'}$ for $z' \neq z$. The vertex $u$ exists since $c_{xz}\neq c_{xz'}$ and $x \in P_{xz'}$ for every $z' \neq z$. Let $v$ be the vertex after $u$ in $P$. By maximality of $u$, the vertex $v$ is in $P_{xz}$ (since $v \in RS(x)$). So if $v \notin PP(x,z)$ then $v \in P_{x'z''}$ for some $x' \neq x$. By Lemma~\ref{nonintche}(a), we have $z=z''$. Thus a vertex of $P_{xz'}$ and a vertex of $P_{x'z}$ are adjacent, contradicting Lemma~\ref{nonintche}(a).

Let $P$ be a path from $x$ to $c^{-}_{xz}$ which does not pass through $PP(x,z)$. Since $Pc^{-}_{xz}c_{xz}$ is a path from $x$ to $c_{xz}$, the first part of the proof ensures that $c_{xz} \in PP(x,z)$. Since $c^{-}_{xz} \notin PP(x,z)$, the lexicographic minimality ensures that $c^{-}_{xz}$ is in $P_{x'z'}$ for $z \neq z'$. Lemma~\ref{nonintche}(a) ensures that $x=x'$. By Observation~\ref{obscritical}, we have $c^{-}_{xz} \in RS(z')$ and $c_{xz} \in RS(z)$, contradicting Lemma~\ref{nonintche}(b).
\end{proof}

Let us finally prove that $X,Z$ is $(2\ell-5)$-disconnectable with the sets $PP(x,z)$. Let $x,z \in X,Z$. Since $X,Z$ is $d$-localized, $P_{xz}$ has length at most $2\ell-7$. In addition $PP(x,z)$ does not intersect $P_{x'z'}$ if $x\neq x'$ or $z \neq z'$; so the deletion of $PP(x,z)$ does not delete all the paths from $x'z'$ of length at most $2\ell-7$. Let us finally show that all the paths of length at most $2\ell-5$ from $x$ to $z$ pass through $PP(x,z)$. 

Since there is no deep escape, any edge leaving $RS(x)$ intersects a critical or a pre-critical vertex. 
By independence, if $z \neq z'$ then we have $d(c_{xz'},z) \geq \ell+1$ and $d(c^{-}_{xz'},z) \geq \ell$. Moreover, we have $d(x,c_{xz'}) =\ell-3$ and $d(x,c^{-}_{xz'})=\ell-4$. Thus the length a path from $x$ to $z$ passing through $c_{xz'}$ or $c^{-}_{xz'}$ is at least $2\ell-4$. Therefore every path of length at most $2\ell-5$ from $x$ to $z$ passes through $c_{xz}$ or $c^{-}_{xz}$. 
By Claim~\ref{relsect}, there is no path of length at most $2\ell-5$ from $x$ to $z$ in $G[V \backslash PP(x,z)]$. 
By Theorem~\ref{bigvc}, the distance VC-dimension is at least $(d+1)$, a contradiction. So case~(1) of Claim~\ref{cliqueorstable} cannot hold, \emph{i.e.} case~(2) holds.
\end{proof}


\subsection{Escape property implies large distance VC-dimension}\label{secvcescape}
The outline of the proof of Lemma~\ref{sgclique} consisted in finding a $(2\ell-5)$-disconnecting pair. The approach is the same when the escape property holds even if the proof is more involved.

\begin{figure}
\centering 
\includegraphics[scale=0.85]{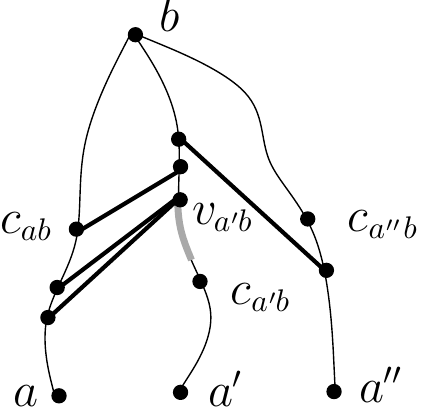}
\caption{The vertex $v$ is the incoming vertex of the $a'b$-path. The gray part (where $v_{a'b}$ is included but not $c_{a'b}$) is the free section of the $a'b$-path.}
\label{fig:firstescape}
\end{figure}

Definitions of this paragraph are illustrated in Figure~\ref{fig:firstescape}. Let $a',b \in A \times B$. The \emph{incoming vertex} $v_{a'b}$ of the $a'b$-path is the first vertex in $P_{a'b}$ (from $a'$ to $b$) for which there exists an escape $u_{a'b}v_{a'b}$ from $a$ to $a'$ for $b$ for some $a \in A$. In other words, it is the first vertex of $P_{a'b}$ at distance one from $RS(A)\setminus RS(a')$. The edge $u_{a'b}v_{a'b}$ is a \emph{first-in escape to $a'$ for $b$}. Note that several first-in escapes to $a'$ can exist, but the incoming vertex is unique. The \emph{free section} of the $a'b$-path, denoted by $FS(a',b)$, is the $c_{a'b}v_{a'b}$-subpath of the $a'b$-path where $c_{a'b}$ is not included but $v_{a'b}$ is included. Lemma~\ref{nonintche}(b) ensures that the free section exists and has length at least $3$.

\begin{lemma}\label{empty1}
Let $A,B$ be a pair satisfying the escape property. Then there is no edge between two free sections of $AB$-paths.
\end{lemma}
\begin{proof}
Consider an edge $xy$ where $x \in FS(a',b')$ and $y \in FS(a,b)$. Let us prove that there is a forbidden cross (see Lemma~\ref{nox}). Notice that $x \in RS(b')$ since $x \in P_{a'b'}$ ($FS(a',b')$ is a subpath of $P_{a'b'}$) and $x \notin RS(a')$ (it is after $c_{a'b'}$). Similarly, $y \in RS(b)$. So Lemma~\ref{nonintche}(b) ensures that $b=b'$. Assume w.l.o.g. that $a<a'$ in the order inherited from $b$. Hence there is a deep escape $uv$ from $a$ to $a'$ for $b$. By definition of deep escape, $y$ is strictly after $c_{ab}$ in $P_{ab}$ and $u$ is strictly before $c^{-}_{ab}$ in $P_{ab}$. So we have $d(u,y) \geq 3$ (since $P_{ab}$ is a minimum path). Moreover $x$ is before $v$ on $P_{a'b}$ by definition of the free section of $P_{a'b}$. Finally edges $xy$ and $uv$ contradict Lemma~\ref{nox}. 
\end{proof}

\begin{figure}
\center\includegraphics[scale=0.7]{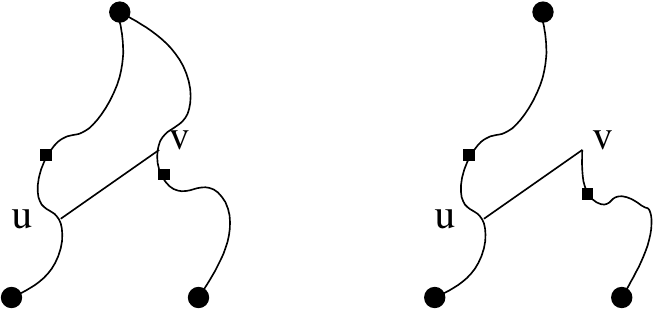}
\caption{The minimum path (at the left) is transformed into the jump path (at the right).}
\label{firstshortcut}
\end{figure}

\begin{lemma}\label{finiz}
The size of a pair with the escape property is at most $2^{d+1}-1$.
\end{lemma}
\begin{proof}
Assume by contradiction that a pair $A,B$ of size $2^{d+1}$ satisfies the escape property. Let $b \in B$. Let us denote by $a_1,\ldots,a_{2^{d+1}}$ the vertices of $A$ ordered along the order inherited from $b$. For every $i \geq 2$, we denote by $v_i$ the incoming vertex and by $u_iv_i$ a first-in escape to $a_i$ for $b$. By convention we put $v_1=b$ and $FS(a_1,b)$ is the subpath of $P_{a_1b}$ from $c_{a_1b}$ to $b$. Recall that there exists $j <i$ such that $u_i \in RS(a_j)$. Note that $v_j$ is after $u_i$ on $P_{a_jb}$. Indeed $u_i$ appears before $c_{a_jb}$ since $u_i \in RS(a_j)$ and $v_j$ appears after $c_{a_jb}$. Therefore the following collection of $Ab$-paths, called jump paths (for $b$), is well-defined:
\begin{itemize}
 \item The \emph{jump path of $a_1b$} is the $a_1b$-path.
 \item The \emph{jump path of $a_ib$} is the $a_iv_i$-subpath of $P_{a_ib}$, the edge $v_{i}u_{i}$ of origin root section $a_j$ and the suffix path on $u_i$ of the jump path of $a_jb$ (see Figure~\ref{firstshortcut}). 
\end{itemize}
Note that jump paths can be equal to minimum paths (see the rightmost part of Figure~\ref{dessinshort}).

Let us analyze a bit the structure of jump paths. The jump path of $a_ib$ starts with the $a_iv_i$-subpath of $P_{a_ib}$. In particular both paths coincide in $RS(a_i)$. Then the jump path of $a_ib$ contains the first in-escape to $a_i$, namely the edge $v_iu_i$. By definition of the order, the vertex $u_i$ is in $RS(a_j)$ for $j<i$ and then $u_iv_i$ is an escape from $a_j$ to $a_i$ for $b$. Thus $u_i$ is in $P_{a_jb} \cap RS(a_j)$. So it is on the jump path of $a_jb$. After this ``rerouting'' the two jump paths are the same and do not quit each other before the end of the path.

Jump paths follow minimum $AB$-paths except on incoming vertices in which they are ``rerouted''. A \emph{rerouting edge} is an edge $e$ such that there exists $i$ satisfying $e=u_iv_i$. Since after a rerouting edge, the jump path of $a_ib$ coincides with the jump path $a_jb$ for $j<i$, every jump path has at most $2^{d+1}$ rerouting edges. Moreover each rerouting edge increases the length of the path by at most two since $|d(u_i,b)-d(v_i,b)| \leq 1$ ($u_iv_i$ is an edge and $P_{a_jb}$ is minimum). Since a pair with the escape property is $d$-localized (each path $P_{ab}$ has length at most $2\ell - 2^{d+2}-3$), the length of the jump path of $ab$ is at most 
\begin{equation}\label{eq1}
(2\ell-2^{d+2}-3)+2^{d+1}\cdot 2 =2\ell-3     
\end{equation}
for every pair $a,b$. 
Let us now state a claim on the structure of the paths.

\begin{claim}\label{formjump}
Any vertex of a jump path is either in $RS(A)$ or in a free section $FS(a,b)$. Moreover any vertex of a jump path for $b$ is in $\bigcup_{a \in A} P_{ab}$.
\end{claim}
\begin{proof}
By induction on the order inherited from $b$. It holds for the jump path of $a_1b$. The jump path of $a_ib$ coincides with the $a_ib$-path from $a$ to the incoming vertex, \emph{i.e.} on $RS(a_i)$ and on $FS(a_i,b)$. By induction, it holds for the remaining vertices since the remaining of the jump path of $a_ib$ is included in the jump path of $a_jb$ for $j<i$.
\end{proof}

In the remaining of the proof we consider the restriction of the graph to the vertices of the jump paths of $ab$ for every $a,b \in A \times B$. Let $a_i \in A, b \in B$. Remind that the first vertex of $FS(a_i,b)$ is the vertex after $c_{a_ib}$ in $P_{a_ib}$ and the last one is $v_i$, the incoming vertex of $P_{a_ib}$.

\begin{claim}\label{degreeemptysection}
Let $i \geq 2$. The vertices of  $FS(a_i,b)$ induce a subpath $w_1,\ldots,w_k=v_i$ of $P_{a_ib}$. The only neighbors of these vertices are the following:
\begin{itemize}
 \item For every $1 \leq q \leq k$, the vertex $w_q$ is incident to $w_{q-1}$ and $w_{q+1}$ (if they exist).
 \item The vertex $w_k=v_i$ has neighbors in $RS(a_j)$ where $a_j$ is the origin root section of $v_i$ (in particular $j<i$ in the order inherited from $b$).
 \item The vertex $w_1$ is incident to $c_{a_ib}$.
\end{itemize}
\end{claim}
\begin{proof}
Claim~\ref{formjump} ensures that every vertex is either in $RS(A)$ or in $FS(A,B)$. By Lemma~\ref{empty1}, there is no edge between two free sections. 
So an edge leaving $FS(a_i,b)$ has an endpoint in $RS(A)$. By definition of incoming vertex, no vertex of $FS(a_i,b)$ distinct from $v_i$ is incident to a vertex of $RS(a_j)$ with $j \neq i$. Moreover, since $P_{a_ib}$ is a minimum path, $w_1$ is the unique vertex of $FS(a_i,b)$ which can be incident to $RS(a_i)$. \\
By definition of $v_i$, there exist edges between $v_i$ and the origin root section of $v_i$, namely $RS(a_j)$. Since every vertex has at most one origin root section, the second point holds. \\
The vertex $w_1$ is incident to $c_{a_ib}$ since they are consecutive in $P_{a_ib}$. Others neighbors of $w_1$ in $RS(a_i)$ must be critical vertices since $d(a_i,w_1)=\ell-2$ (indeed $d(c_{a_ib},a_i)=\ell-3$ and $P_{a_ib}$ is minimum). Thus the matching property of Observation~\ref{obscritical2} ensures that $w_1$ has no other neighbor in $RS(a_i)$, which concludes the proof of Claim~\ref{degreeemptysection}.
\end{proof}

In particular, Claim~\ref{degreeemptysection} ensures that any path $P$ leaving $FS(a_i,b)$ has to enter in $RS(a_i)$ or in $RS(a_j)$. Conversely, you can notice that any neighbor of a vertex in $RS(a_i)$ is either in $RS(a_i)$ or is in some $FS(a_j,b')$ for $j>i$. These two observations are the most important pieces of the proofs of the remaining statements.

Remind that any path of length at most $2\ell-3$ from $a$ to $b$ does not pass through $c_{a'b'}$ with $b \neq b'$. Indeed by independence, $d(a,c_{a'b'})\geq \ell-3$ and $d(b,c_{a'b'})>\ell$. 

\begin{claim}\label{clm:nosuprs}
 Any path of length at most $2\ell-3$ from $a_i$ to $b$ does not contain any vertex in $RS(a_j)$ for $j>i$ (in the order inherited from $b$).
\end{claim}
\begin{proof}
Assume by contradiction that such a path $P$ exists and denote by $j$ the maximum index such that $P$ passes through $RS(a_j)$. Note that $j \geq 2$. Let $u$ be the first vertex of $P$ in $RS(a_j)$ and let $v$ be the vertex before $u$ in $P$. The path $P$ cannot enter in $RS(a_j)$ through $c_{a_jb'}$ with $b'\neq b$ since $P$ has length at most $2\ell-3$. 
Lemma~\ref{nonintche}(a) ensures that $v \notin RS(A)$. So Claim~\ref{formjump} ensures that $v \in FS(a_k,b')$.

Assume first that $a_k \neq a_j$. Since $u \in RS(a_j)$, $uv$ is an escape from $a_j$ to $a_k$ for $b'$. In particular, it means that $k>j$. Since $u$ is the first vertex of $P$ in $RS(a_j)$, the path $P$ cannot enter in $FS(a_k,b')$ through $RS(a_j)$. So Claim~\ref{degreeemptysection} ensures that $P$ enters in $FS(a_k,b')$ through $RS(a_k)$, contradicting the maximality of $j$.

Assume now that $a_k=a_j$. Claim~\ref{degreeemptysection} ensures that the unique vertex of $FS(a_j,b')$ with a neighbor in $RS(a_j)$ is the first vertex of $FS(a_j,b')$, so $v$ is this vertex. Moreover, the unique neighbor of $v$ in $RS(a_j)$ is the vertex $c_{a_jb'}$ by Claim~\ref{degreeemptysection}. Since $P$ cannot pass through $c_{a'b'}$ with $a' \neq a_i$ and $b' \neq b$, we have $b'=b$. So $u=c_{a_jb}$ and $v$ is the first vertex of $FS(a_j,b)$. 
Let us now denote by $w$ the last vertex of $P$ in $RS(a_j)$. Note that $w \neq c_{a_jb'}$ for $b' \neq b$. Moreover the vertex after $w$ in $P$ cannot be in $FS(a_j,b)$ since otherwise this vertex would be $v$, and then $P$ would not be a path ($v$ would appear twice in $P$). So the edge used to live $RS(a_j)$ is an escape to $a_\ell$ for $b''$. In particular, $\ell > j$. By Claim~\ref{degreeemptysection}, vertices of $FS(a_\ell,b'')$ only have neighbors in $RS(a_j)$ and in $RS(a_\ell)$. Since $w$ is the last vertex in $RS(a_j)$, when $P$ leaves $FS(a_\ell,b')$ it enters in $RS(a_\ell)$, contradicting the maximality of $j$.
\end{proof}

\begin{claim}\label{passecxz}
The vertex $c_{ab}$ is in every path $P$ from $a$ to $b$ of length at most $2\ell-3$. Moreover if a vertex of the $ac_{ab}$-subpath of $P$ is not in $RS(a)$, then the next one is.
\end{claim}
\begin{proof}
Let $P$ be a path from $a$ to $b$ of length at most $2\ell-3$. Let $u$ be the last vertex of $u$ in $RS(a)$. Let $v$ be the vertex after $u$ in $P$. For distance reasons, $P$ the vertex $u$ is not $c_{ab'}$ for $b' \neq b$. Let us show that $P$ does not leave $RS(a)$ using an escape. Assume by contradiction that $v$ is in $FS(a_j,b')$ for $a_j>a$ (in the order inherited from $b'$). Let us denote by $w$ the first vertex of $P$ after $v$ which is not in $FS(a_j,b')$. By Claim~\ref{degreeemptysection}, $w$ is either in $RS(a)$ or is $c_{a_jb'}$. Since $w$ is after $u$ in $P$, $w \notin RS(a)$, so $w=c_{a_jb'}$. If $b \neq b'$, we have a distance contradiction since both $a$ and $b$ are at distance more than $\ell$ from $c_{a_jb'}$. If $b=b'$, then $a_j > a$ in the order inherited from $b$, contradicting Claim~\ref{clm:nosuprs}.

So the vertex $u$ is the vertex $c_{ab}$. In addition, in the $ac_{ab}$-subpath of $P$, if a vertex is not in $RS(a)$, then it is in $FS(a_j,b)$ where $a_j>a$. Claims~\ref{degreeemptysection} and~\ref{clm:nosuprs} ensure that the next vertex is in $RS(a)$.
\end{proof}

The \emph{jump private part} of $a$ and $b$, denoted by $JPP(a,b)$, is the set of vertices which are in the jump path of $ab$ and in no other jump path.

\begin{claim}\label{passucs}
All the paths of length at most $2\ell-3$ from $a$ to $b$ pass through $JPP(a,b)$.
\end{claim}
\begin{proof}
\begin{figure}
 \centering
 \includegraphics{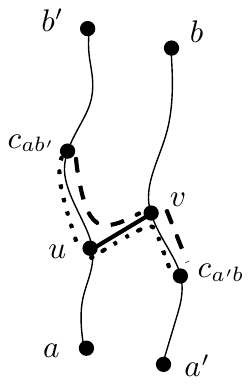}
 \caption{Illustration of Claim~\ref{passucs}. The two dotted paths are the two sides of an inequality. And the two dashed paths are the two sides of the other one.}
 \label{figlastclaim}
\end{figure}
Let $P$ be a path from $a$ to $b$ of length at most $2\ell-3$. 
Claim~\ref{passecxz} ensures that $P$ passes through $c_{ab}$. Assume by contradiction that the subpath of $P$ between $a$ and $c_{ab}$ does not pass through $JPP(a,b)$. 
Let $u$ be the last vertex of the $ac_{ab}$-subpath of $P$ which is in the path $P_{ab'}$ for $b' \neq b$. Such a vertex exists since $c_{ab}$ is not in $P_{ab'}$ for $b' \neq b$ by Observation~\ref{obscritical2}.
Let $v$ be the vertex after $u$ on $P$. And, for every $b' \in B$, $a$ is in the path $P_{ab'}$.

If $v \notin RS(a)$ then Claim~\ref{passecxz} ensures that the vertex after $v$ is in $RS(a)$. So $v$ is in $P_{a'b''}$ with $a' \neq a$. The vertex after $v$ is in $P_{ab''}$ since it is in $RS(a)$. By maximality of $u$, we have $b=b''$. Thus $u \in P_{ab'}$ for $b' \neq b$ (by definition of $u$) and $v \in  P_{a'b}$ for $a' \neq a$, a contradiction with Lemma~\ref{nonintche}(a).

So $v \in RS(a)$ and then $v \in P_{ab}$. Assume by contradiction that $v \notin JPP(a,b)$. So the vertex $v$ is in the jump path of $a'b$ for some $a' \neq a$. Free to modify $a'$, we may assume that the jump path of $a'b$ has been rerouted only once before $v$.
The vertex $v$ is on the $c_{a'b}b$-subpath of the jump path of $a'b$ and $u$ is on the $ac_{ab'}$-subpath of $P_{ab'}$. The two following inequalities, illustrated on Figure~\ref{figlastclaim}, provide a contradiction. \\
First $d(u,c_{ab'}) +3 < d(v,c_{a'b})+1$ since $d(a,c_{ab'}) \leq \ell-3$ and $d(a,c_{a'b})>\ell$. Indeed, by definition of critical vertex, $d(a,c_{ab'}=\ell-3$ (even in this induced subgraph) and $d(a,c_{a'b})>\ell$ is a consequence of the independence. Since $u$ is on a minimum $ac_{ab'}$-path, the inequality holds. \\
Second $d(v,c_{a'b}) < d(u,c_{ab'})+1$ since $d(b,c_{a'b}) \leq \ell$ and $d(b,c_{ab'})>\ell$ and $uv$ is an edge.
The first inequality is due to the fact that jump paths have length at most $2\ell-3$ and that the length of the $a'c_{a'b}$-subpath of the jump path of $a'b$ is exactly $\ell-3$. The second inequality is a consequence of the independence of $A,B$. \\
The sum of these two inequalities gives $3<2$, a contradiction.
\end{proof}

To conclude the proof of Lemma~\ref{finiz}, we apply Theorem~\ref{bigvc} with the sets $JPP(a,b)$ for paths of length at most $2\ell-3$. Equation~(\ref{eq1}) ensures that the jump path of $xz$ has length at most $2\ell-3$, so if $JPP(a,b)$ is not selected, there remain paths of length at most $2\ell-3$. The sets $JPP(x,z)$ are pairwise disjoint and are only on the jump path of $xz$. Claims \ref{passecxz} and \ref{passucs} ensure that the sets $JPP(x,z)$ are $(2\ell-3)$-disconnecting for $X,Z$. So the graph $G$ has distance VC-dimension at least $d+1$, a contradiction.
\end{proof}
By combining Theorem~\ref{Matou} and Lemmas~\ref{chemincourtlong},~\ref{depol},~\ref{sgclique},~\ref{finiz}, we obtain Theorem~\ref{erdpo}.


\section{Concluding remarks}
In Section~\ref{VC}, we did not make any attempt to improve the gap function. We made exponential extractions at several steps as Ramsey's extractions and the function of Theorem~\ref{Matou} is not expressed in the original paper of Matou\v{s}ek. Finding a polynomial gap instead of an exponential one is an interesting problem, though probably a hard one.
We can also study this problem for particular classes of graphs. Chepoi, Estellon and Vax\`es conjectured that the gap function between $\nu_{\ell}$ and $\tau_{\ell}$ for planar graphs is linear. More formally they conjectured the following.
\begin{conjecture}{(Chepoi, Estellon, Vax\`es \cite{ChepoiEV07})}
There exists a constant $c$ such that $\tau_{\ell}(G) \leq c \cdot \nu_{\ell}(G)$ for every $\ell$ and every planar graph $G$.
\end{conjecture}
Dv\v{o}r\'ak proved in~\cite{Dvorak11} that $\tau_\ell \leq c(\ell) \nu_\ell$ for bounded expansion classes. Moreover the function $c$ is a polynomial function. 

In graph coloring, we need some structure to bound the chromatic number. The \emph{chromatic number} $\chi(G)$ of the graph $G$ is the minimum number of colors needed to color \emph{properly} the vertices of $G$, \emph{i.e} such that two adjacent vertices of $G$ receive distinct colors. The size of the maximum clique of $G$, denoted by $\omega(G)$, is a lower bound on the chromatic number $\chi(G)$. The gap between $\chi$ and $\omega$ can be arbitrarily large since there exist triangle-free graphs with an arbitrarily large chromatic number (Erd\H{o}s was the first to construct some of them in~\cite{Erdos59}). A class of graphs $\mathcal{C}$ is \emph{$\chi$-bounded} if there exists a function $f$ such that for every graph $G \in \mathcal{C}$, every induced subgraph $G'$ of $G$ satisfies $\chi(G') \leq f(\omega(G'))$. Dv\v{o}r\'ak and Kr\'a\soft{l} proved in \cite{dvorakk11} that graphs of bounded rankwidth are $\chi$-bounded. Actually they proved it for classes of graphs with cuts of small rank. Since the distance VC-dimension catches the complexity of the intersection of neighborhoods at large distance, the same might be extended for graphs of bounded distance VC-dimension.
\begin{conjecture}
Let $\mathcal{G}$ be a class of graphs. If there exists a function $f$ such that the distance VC-dimension of $G \in \mathcal{G}$ is at most $f(\omega(G))$ then $\mathcal{G}$ is $\chi$-bounded.
\end{conjecture}
We also conjecture that the following graph classes, known to be $\chi$-bounded, have a bounded distance VC-dimension.
\begin{conjecture}
The distance VC-dimension of every $P_{\ell}$-free graph $G$ is bounded by a function of $\ell$ and $\omega(G)$. 
Similarly the distance VC-dimension of every circle graph $G$ is bounded by a function of $\omega(G)$.
\end{conjecture}
\vspace{18pt}

\noindent\textbf{Acknowledgement:} The authors want to thank the anonymous reviewers for fruitful comments. We also want to thanks L\'aszl\'o Kozma and Shay Moran for helpful remarks about Theorem~\ref{Matou}.

\bibliographystyle{plain}

\end{document}